\documentclass[11pt]{amsart}
\usepackage{amsmath,amssymb,amscd,amsfonts,verbatim}
\usepackage{paralist}
\usepackage[mathscr]{eucal}

\newtheorem{thm}{Theorem}[subsection]
\newtheorem{lem}[thm]{Lemma}

\newtheorem{prop}[thm]{Proposition}

\newtheorem{conj}[thm]{Conjecture}

\newtheorem{assu-nota}[thm]{Assumption--Notation}
\theoremstyle{remark}
\newtheorem{example}[thm]{Example}
\newtheorem{remark}[thm]{Remark}
\newcommand{\pionealg}{\pi_1^{\mathrm{alg}}}
\newcommand{\iso}{\cong}

\newcommand{\ol}{\overline}
\newcommand{\C}{\mathbb C}
\newcommand{\Z}{\mathbb Z}
\newcommand{\Q}{\mathbb Q}

\newcommand{\pp}{\mathbb P}

\newcommand{\btimes}{\,{\scriptstyle \boxtimes}\,}
\DeclareMathOperator{\albdim}{Albdim}
\DeclareMathOperator{\Alb}{Alb}

\DeclareMathOperator{\Hom}{Hom}

\DeclareMathOperator{\Pic}{Pic}

\DeclareMathOperator{\Spec}{Spec}

\newcommand{\al}{\alpha}
\newcommand{\be}{\beta}
\newcommand{\ga}{\gamma}
\newcommand{\la}{\lambda}
\newcommand{\Ga}{\Gamma}
\newcommand{\De}{\Delta}

\newcommand{\si}{\sigma}

\newcommand{\cM}{\mathcal M}

\newcommand{\OO}{\mathcal{O}}

\newcommand{\inv}{^{-1}}
\newcommand{\unu}{^{\nu}}

\numberwithin{equation}{subsection}
\usepackage[dvips,colorlinks=false]{hyperref}

\title{The geography of irregular surfaces}
\author{Margarida Mendes Lopes and Rita Pardini}

\thanks{The first author is a member of the Center for Mathematical
Analysis, Geometry and Dynamical Systems  and the second author is a member of G.N.S.A.G.A.-I.N.d.A.M. This research was partially supported by  the Funda\c c\~ao para a Ci\^encia e Tecnologia (FCT- Portugal).} \begin{document}
\begin{abstract}
We give an overview  of irregular complex surfaces of general type, discussing in particular the distribution of 
the numerical invariants  $K^2$, $\chi$ for  the minimal  ones. \\
This is an expanded version of the talk given by the second author at the workshop ``Classical Algebraic Geometry Today'',
M.S.R.I., January 26--30, 2009.\\

\end{abstract}
\maketitle
\tableofcontents

\section{Introduction}

Let $S$ be a minimal surface of general type and let $K^2, \chi$ be its main  numerical invariants (cf. \S \ref{ssec:invariants}). For every pair of positive integers  $a,b$ the surfaces with $K^2=a$, $\chi=b$ belong to finitely many irreducible families, so that in principle their classification is possible. In practice, the much weaker {\em geographical problem}, i.e. determining  the  pairs $a,b$ for which there exists a minimal surface of general type  with $K^2=a$, $\chi=b$, is quite hard. 

In the  past, the main focus in the study of  both the geographical problem and the fine classification of surfaces of general type has been on {\em regular} surfaces, namely surfaces that have no global 1-forms, or, equivalently, whose first Betti number is 0. The reason for this is twofold: on one hand,   the canonical map of regular surfaces is easier to understand, on the other hand  complex surfaces are the main source of examples of differentiable 4-manifolds, hence the  simply connected ones  are considered especially interesting from that point of view.

So, while, for instance, the geographical problem is by now almost settled and the fine classification of some families of regular surfaces is accomplished, little is known about irregular surfaces of general type. In recent years, however, the use of new methods and the revisiting of old ones have produced several new  results.

Here we give an overview  of these results,  with special emphasis on the geographical problem.  In addition, we give several examples and discuss some open questions and  possible generalizations to higher dimensions (e.g.,  Thm. \ref{thm:Severi_3folds}). 
\smallskip

{\bf Notation and conventions.}  We work over the complex numbers. 
All varieties are projective algebraic and, unless otherwise specified, smooth.
We denote by $J(C)$ the Jacobian of a curve $C$.

Given varieties $X_i$ and sheaves $\mathcal F_i$ on $X_i$, $i=1,2$,  we denote by $\mathcal F_1\btimes \mathcal F_2$ the sheaf $p_1^*{\mathcal F_1}\otimes p_2^*\mathcal F_2$ on $X_1\times X_2$, where $p_i\colon X_1\times X_2\to X_i$ is the projection.

\section{Irregular surfaces of general type}

Unless otherwise specified, a surface is a smooth projective complex surface. 

A surface $S$ is {\em of general type} if the canonical divisor $K_S$ is big. Every surface of general type has a birational morphism onto a unique minimal model, which is characterized by the fact that $K_S$ is nef. 
A surface (or more generally a variety) $S$ is called {\em irregular} if the {\em irregularity} $q(S):=h^0(\Omega^1_S)=h^1(\OO_S)$ is $>0$.

 \subsection{Irrational pencils}\label{ssec:pencils}

A {\em pencil} of a surface $S$ is a morphism with connected fibres $f\colon S\to B$, where $B$ is a smooth curve. A map $\psi\colon S\to X$, $X$ a variety, is {\em composed with a pencil} if there exists a pencil $f\colon S\to B$ and a map $\ol \psi\colon B \to X$ such that $\psi=\ol\psi \circ f$.
 The {\em genus} of the pencil $f$ is by definition the genus $b$ of $B$.  The pencil $f$ is {\em irrational} if $b>0$. Since pull back of forms induces an injective map $H^0(\Omega^1_B)\to H^0(\Omega^1_S)$, a surface with an irrational pencil is irregular. Clearly, the converse is not true (cf. Remark \ref{rem:pencils}).

 In addition, if the pencil $f$ has genus $\ge 2$, by   pulling back two independent $1$-forms of $B$ one gets independent $1$-forms $\al$ and $\be$ on $S$ such that $\al\wedge \be=0$. The following classical result  (cf.  \cite{beauville_libro} for a proof)  states that this condition is equivalent to the existence of an irrational pencil of genus $\ge 2$:
 \begin{thm}[Castelnuovo--De Franchis]\label{thm:CDF}
 Let $\al, \be\in H^0(\Omega^1_S)$ be linearly independent forms such that $\al\wedge\be=0$. Then there exists a pencil $f\colon S\to B$ of genus $\ge 2$ and $\al_0, \be_0\in H^0(\omega_B)$ such that $\al=f^*\al_0$, $\be=f^*\be_0$.
 \end{thm} 
 
 Let $\al_1, \dots \al_k, \be \in H^0(\Omega^1_S)$ be linearly independent forms such  that $\al_1, \dots \al_k$ are pull backs from a curve  via a pencil $f\colon S \to B$ and $\be\wedge \al_j=0$ (notice that it is the same to require this for one index $j$ or for all $j=1,\dots k$). By the Theorem of Castelnuovo--De Franchis (Theorem \ref{thm:CDF}), there exists a pencil  $h\colon S\to D$  such that, say, $\al_1$ and $\be$ are pull backs of independent  $1$-forms of $D$.  Let $\psi:=f\times h\colon S\to B\times D$. Then the forms $\al_j\wedge \be$ are pull backs of nonzero $2$-forms of $B\times D$. Since $\al_j\wedge \be=0$, it follows that the image of $\psi$ is a curve $C$ and that $\al_1,\dots \al_k, \be$ are pullbacks of $1$-forms of $C$. This shows that  for every $b\ge 2$ there is 
  one-to-one correspondence between pencils of $S$ of genus $b$ and subspaces $W\subset H^0(\Omega^1_S)$ of dimension $b$  such that $\wedge^2 W=0$ and $W$ is maximal with this property.
 
The existence of a subspace $W$ as  above can be interpreted in terms of  the cohomology of $S$ with complex coefficients, thus showing that the existence of a pencil of genus $\ge 2$ is a topological property (cf. \cite[Thm. 1.10]{catanese_irregular}). On the contrary, the existence of pencils of genus 1 is not detected by the topology (cf. Remark \ref{rem:pencils}).

A classical result of Severi (cf. \cite{Severi}, \cite{samuel}) states that a surface of general type has finitely many pencils of genus $\ge 2$.
However a surface of general type can have infinitely many pencils of genus 1, as it is shown by the following example:
\begin{example}
Let $E$ be a curve of genus 1, $O\in E$ a point,  $A:=E\times E$ and  $L:=\OO_A(\{O\}\times E+E\times \{O\})$. For every $n\ge 1$, the map $h_n\colon A\to E$ defined by $(x,y)\mapsto x+ny$ is a pencil of genus 1.  Let now $S\to A$ be a double cover branched on a smooth ample curve $D\in |2dL|$ (see \S \ref{ssec:constructions}, (d) for a quick review of double covers).  The surface $S$ is minimal of general type, and for every $n$ the map $h_n$ induces a pencil $f_n\colon S\to E$. The general fiber of $f_n$ is a double cover of an elliptic curve isomorphic to $E$, branched on $2d(n^2+1)$ points, hence it has genus $dn^2+d+1$. It follows that the pencils $f_n$ are all distinct.
\end{example}

\subsection{The Albanese map}\label{ssec:albanese}

 The {\em Albanese variety\,}  of $S$ is defined as $\Alb(S)\!:=H^0(\Omega^1_S)^{\vee}/H_1(S,\Z)$. By Hodge theory, $\Alb(S)$ is a compact complex torus and, in addition,  it can be embedded in projective space, namely it is an abelian variety. For a fixed base point $x_0\in S$ one defines  the Albanese morphism $a_{x_0}\colon S\to \Alb(S)$ by $x\mapsto \int_{x_0}^x\!\!\!-$ (see \cite[Ch.V]{beauville_libro}). Choosing a  different base point in $S$, the Albanese morphism just changes by a translation of $\Alb(S)$, so we often  ignore the base point and just  write $a$.
By construction, the map $a_*\colon H_1(S,\Z)\to H_1(\Alb(S),\Z)$ is surjective, with kernel equal to the torsion subgroup of $H_1(S,\Z)$, and the map $a^*\colon H^0(\Omega^1_{\Alb(S)})\to H^0(\Omega^1_S)$ is an isomorphism, so if $q(S)>0$ it follows immediately that $a$ is non constant. The dimension of $a(S)$ is called the {\em Albanese dimension} of $S$ and it
 is denoted by $\albdim(S)$. $S$ is of {\em Albanese general type} if $\albdim(S)=2$ and $q(S)>2$.

The morphism $a\colon S\to \Alb(S)$ is characterized  up to unique  isomorphism by the following universal property:
for every morphism $S\to T$, with $T$ a complex torus,  there exists a unique factorization $S\xrightarrow{a} \Alb(S)\to T$.
It follows immediately that the image of  $a$ generates $\Alb(S)$, namely that  $a(S)$ is not contained  in any proper subtorus of $\Alb(S)$.
Using the Stein factorization and the fact that for a smooth curve $B$  the Abel--Jacobi map $B\to J(B)$ is an embedding, one  can show  that  if $\albdim(S)=1$, then $B:=a(S)$ is a smooth curve  of genus $q(S)$ and the map $a\colon S\to B$ has connected fibers. In this case,  the map $a\colon S\to B$ is called the {\em Albanese pencil} of $S$.
By the analysis of \S \ref{ssec:pencils},  $\albdim(S)=1$ (``$a$ is composed with a pencil'')  iff  $\al\wedge \be=0$ for every pair of $1$-forms $\al,\be\in H^0(\Omega^1_S)$ iff there exists a  surjective map $f\colon S\to B$ where $B$ is a smooth curve of genus $q(S)$. In this case,  if $q(S)>1$ then $f$ coincides with the Albanese pencil.

Since, as we  recalled in \S \ref{ssec:pencils}, the existence of a pencil of given genus $\ge 2$ is a topological property of $S$, the Albanese dimension of a surface is a topological property.

\begin{remark}\label{rem:pencils} Let   $f\colon S\to B$ be  an irrational pencil. By the universal property of the Albanese variety, there is a morphism of tori $\Alb(S)\to J(B)$. The differential of this morphism at $0$ is dual to $f^*\colon H^0(\omega_B)\to H^0(\Omega^1_S)$, hence it is surjective and  $\Alb(S)\to J(B)$ is surjective, too. So if $\albdim (S)=2$ and $\Alb(S)$ is simple, $S$ has no irrational pencil. It is easy to produce examples of this situation, for instance by considering surfaces that are complete intersections inside a simple abelian variety (cf. \S \ref{ssec:constructions}, (c)).

If $\albdim(S)=q(S)=2$, $S$ has an irrational pencil iff there exists a surjective map $\Alb(S)\to E$, where $E$ is an elliptic curve.
So, by taking double covers of principally polarized abelian surfaces branched on a smooth ample curve (cf. \S \ref{ssec:constructions}, (d)) one can construct a family of minimal surfaces of general type such that the general surface in the family  has no irrational pencil  but  some special ones have.
 \end{remark}
If $\albdim(S)=2$, then $a$ contracts finitely many irreducible curves.  By Grauert's criterion (cf. \cite[Thm.III.2.1]{bpv}), the intersection matrix  of the set of curves contracted by $a$ is negative definite.
 An irreducible curve $C$ of $S$ is contracted by $a$ if and only if the restriction map $H^0(\Omega^1_S)\to H^0(\omega_{C\unu})$ is trivial,  where $C\unu \to C$ is the normalization. So, every rational curve of $S$ is contracted by $a$. More generally, by the universal property the map $C\unu \to \Alb(S)$ factorizes through $J(C\unu)=\Alb(C\unu)$, hence $a(C)$  spans an abelian subvariety of $\Alb(C)$ of dimension at most $g(C\unu)$. So, for instance, if $\Alb(S)$ is simple every curve of $S$ of geometric genus $<q(S)$ is contracted by $a$.
 
An irreducible curve $C$   of $S$ such that $a(C)$ spans a proper abelian subvariety $T\subset \Alb(S)$ has $C^2\le 0$.  In particular, if $C$ has geometric genus $<q(S)$ then  $C^2\leq 0$.  Indeed, consider the non constant  map $\overline{a}\colon S\to A/T$ induced by $a$. If the image of $\overline{a}$ is a surface, then $C$  is  contracted to a point,  hence $C^2<0$. If  the image of $\overline{a}$ is a curve, then $C$ is contained in a fiber, hence by Zariski's lemma (\cite[Lem.III.8.9]{bpv}) $C^2\le 0$  and $C^2=0$ iff $C$ moves in an irrational pencil. 
In view of the fact that $S$ has finitely many irrational pencils of genus $\ge 2$, the irreducible curves of $S$ whose image via $a$ spans an abelian subvariety of $\Alb(S)$ of codimension $>1$  belong to finitely many numerical equivalence classes.

We close this section by giving an example that shows that the degree of the Albanese map of a surface with $\albdim =2$ is not a topological invariant.
\begin{example}\label{ex:albdeg}
We describe  an irreducible  family  of smooth minimal  surfaces of general type   such that the 
 Albanese map of the general element of the family is generically  injective but for some special elements  the Albanese map
 has degree 2 onto its image. 
The examples are constructed as divisors in a double cover $p\colon V\to A$ of an abelian
threefold  $A$.

Let $A$ be an abelian threefold and let $L$ be an ample line bundle of $A$ such that $|2L|$ contains a smooth divisor $D$. There is a double cover $p\colon V\to A$   branched on $D$ and such that $p_*\OO_V=\OO_A\oplus L\inv$ (cf. \S \ref{ssec:constructions}, (d)). The variety $V$ is smooth and, arguing as in \S \ref{ssec:constructions}, (d),  in one shows that  the Albanese map of $V$ coincides with $p$. Notice that this implies that the map 
$p_*\colon H_1(V,\Z)\to H_1(A,\Z)$ is an isomorphism up to torsion (cf. \S \ref{ssec:albanese}). 

Let now $Y\subset A$ be a  very ample divisor such that $h^0(\OO_A(Y)\otimes L\inv)>0$ and set $X:=p^*Y$. If $Y$
is general, then both $X$ and $Y$ are smooth.
Let now $X'\in |X|$ be a smooth element. By the adjunction formula, $X'$ is smooth of general type.  By the Lefschetz theorem on hyperplane sections, the
inclusion
$X'\to V$ induces an isomorphism $\pi_1(X')\simeq \pi_1(V)$, which in turn gives an isomorphism 
$H_1(X',\Z)\simeq H^1(V,\Z)$. Composing with $p_*\colon H_1(V,\Z)\to H_1(A,\Z)$ we get an
isomorphism (up to torsion) $H_1(X',\Z)\to H_1(A,\Z)$, which is induced by the map $p|_{X'}\colon
X'\to A$. Hence $p|_{X'}\colon
X'\to A$ is the Albanese map of $X'$. Now the map $p|_{X'}$ has degree 2 onto its image if $X'$
is invariant under the involution $\si$ associated to $p$, and it is generically injective
otherwise.  By the projection formula for double covers, the general element of $|X|$ is not
invariant under $\si$ iff $h^0(\OO_A(Y)\otimes L\inv)>0$, hence we have the required example.
\end{example}

\subsection{Numerical invariants and geography}\label{ssec:invariants}
To a minimal complex surface $S$ of general type, one can attach several  integer invariants, besides the irregularity $q(S)=h^0(\Omega^1_S)>0$:
\begin{itemize}
\item  the self intersection $K^2_S$ of the canonical class,
\item the {\em geometric genus}  $p_g(S):=h^0(K_S)=h^2(\OO_S)$,
\item the {\em holomorphic  Euler--Poincar\'e characteristic}, $\chi(S):=h^0(\OO_S)-h^1(\OO_S)+h^2(\OO_S)=1-q(S)+p_g(S)$,
\item the second Chern class $c_2(S)$ of the tangent bundle, which coincides with the topological Euler characteristic of $S$.
\end{itemize}

All these invariants are determined by the topology of $S$ plus the orientation induced by the complex structure. Indeed (cf. \cite[I.1.5]{bpv}), by Noether's formula we have:
\begin{equation}
K^2_S+c_2(S)=12\chi(S), 
\end{equation}
and by the Thom--Hirzebruch index theorem:
\begin{equation}
\tau(S)=2(K^2_S-8\chi(S)),
\end{equation}
where $\tau(S)$ denotes the index of the intersection form on $H^2(S,\C)$, namely the difference between the number of positive and negative eigenvalues. So $K^2_S$  and $\chi(S)$ are determined by the (oriented)  topological invariants $c_2(S)$ and $\tau(S)$.
By Hodge theory the irregularity $q(S)$ is equal to $\frac 12b_1(S)$, where $b_1(S)$ is the first Betti number. So $q(S)$ is also  determined by the topology of $S$ and the same is true for $p_g(S)=\chi(S)+q(S)-1$.

It  is apparent from the definition that these  invariants are not independent. So it is usual to take $K^2_S$, $\chi(S)$ (or, equivalently, $K^2_S$, $c_2(S)$) as the main numerical invariants. These determine the Hilbert polynomial of the $n$-canonical image of $S$ for $n\ge 2$, and by a classical result (cf. \cite{gieseker}) the coarse moduli space $\cM_{a,b}$ of surfaces of general type with $K^2=a$, $\chi=b$ is a quasi projective variety. Roughly speaking, this means that surfaces with fixed  $K^2$ and $\chi$ are parametrized by a finite number of irreducible varieties, hence in principle they can be classified. In practice, however, the much more basic {\em geographical question}, i.e., ``for what values of $a, b$ is $\cM_{a,b}$ nonempty?'' is already  non trivial.

The invariants $K^2,\chi$ are subject to the following restrictions:
\begin{itemize}
\item $K^2$, $\chi>0$,
\item $K^2\ge 2\chi-6$ (Noether's inequality),
\item $K^2\le 9\chi$ (Bogomolov--Miyaoka--Yau inequality).
\end{itemize}
All these inequalities are sharp and it is known that for ``almost all''  $a,b$ in the admissible range the space $\cM_{a,b}$ is nonempty. (The possible exceptions seem to be due to the method of proof and not to the existence of special areas  in the admissible region for the invariants of surfaces of general type).

 In this note we focus on the geographical question  for irregular surfaces. More precisely, we address the following questions:
 
  ``for what values of $a,b$ does there exist a minimal surface of general type $S$ with $K^2=a$, $\chi=b$ such that:
\begin{itemize}
\item $q(S)>0$?''
\item $S$ has an irrational pencil?''
\item $\albdim(S)=2$?''
\item $q(S)>0$ and $S$ has no irrational pencil?"
 \end{itemize}
\begin{remark}
If $S'\to S$ is an \'etale cover of degree $d$ and $S$ is minimal of general type, $S'$ is also minimal of general type with invariants $K^2_{S'}=dK^2_S$ and $\chi(S')=d\chi(S)$. The first  three properties listed above are stable under \'etale covers.
Since the first Betti number of a surface is equal to $2q$, an irregular surface has \'etale covers of degree $d$ for any  $d>0$.
 Hence if   for some $a, b$ the answer to one of these three questions is ``yes'' , the same is true for all the pairs $da, db$, with $d$ a positive integer.
\end{remark}

The main known  inequalities for the invariants  of irregular surfaces of general type  are illustrated in the following sections. Here we only point out  the following simple consequence of Noether's inequality:
\begin{prop}\label{prop:2chi}
Let $S$ be a minimal irregular surface of general type. Then:
$$K^2_S\ge 2\chi(S).$$
\end{prop}
\begin{proof} Assume for contradiction that $K^2_S<2\chi(S)$. Then an \'etale cover $S'$ of degree $d\ge 7$ has $K^2_{S'}<2\chi(S')-6$, violating  Noeher's inequality.
\end{proof}

More generally, the following inequality holds for minimal irregular surfaces of general type (cf. \cite{Debarre_inequalities}): \begin{equation}\label{eq:debarre}
K^2\ge \max\{2p_g, 2p_g+2(q-4)\}.
\end{equation}
The inequality \eqref{eq:debarre} implies  that irregular surfaces with $K^2=2\chi$ have $q=1$.  These surfaces are described in \S \ref{ssec:examples}, (b).

\subsection{Basic constructions}\label{ssec:constructions}
Some constructions of irregular surfaces of general type have already been presented in the previous sections. We list   and describe briefly   the most standard ones:

(a) \underline{Products of curves.\/} Take  $S:=C_1\times C_2$ , with $C_i$ a curve  of genus $g_i\ge 2$. $S$  has  invariants:
$$K^2=8(g_1-1)(g_2-1), \quad \chi=(g_1-1)(g_2-1),\quad  q=g_1+g_2,\quad p_g=g_1g_2.$$
In particular these surfaces satisfy $K^2=8\chi$. The Albanese variety is the product $J(C_1)\times J(C_2)$ and the Albanese map induces an isomorphism onto its image. The two projections $S\to C_i$ are pencils of genus $g_i\ge 2$.
\medskip

(b) \underline{Symmetric products.\/} Take $S:=S^2C$, where $C$ is a smooth curve of genus $g\ge 3$.  Consider the natural map $p\colon C\times C\to S$, which is the quotient map by the involution $\iota$ that exchanges the two factors of $C\times C$.
The ramification divisor of $p$ is the diagonal $\De\subset C\times C$, hence we have:
$$K_{C\times C}=p^*K_{S}+\Delta.$$
Computing intersections on $C\times C$ we get:
$$K^2_S=(g-1)(4g-9).$$
 Global $1$ and $2$-forms on $S^2C$ correspond to forms on $C\times C$ that are invariant under $\iota$.  Writing down the action of $\iota$ on $H^0(\Omega^i_{C\times C})$, one obtains   canonical identifications:
\begin{equation}\label{eq:symmetric}
H^0(\omega_{S})=\wedge^2H^0(\omega_C), \quad H^0(\Omega^1_{S})=H^0(\omega_C).
\end{equation}
Thus  we have:
 $$p_g=g(g-1)/2,\quad q=g,\quad \chi=g(g-3)/2+1.$$
  Since $p_g(S)>0$ and $K^2_S>0$, it follows that $S$ is of general type. 
Notice that by Theorem \ref{thm:CDF}  $S$ has no irrational pencil of genus $\ge 2$, since by \eqref{eq:symmetric} the natural map $\wedge^2H^0(\Omega^1_S)\to H^0(\omega_S)$ is injective.

The points of $S$ can be identified with the effective divisors of degree 2 of $C$. If $C$ is hyperelliptic, then the $g^1_2$ of $C$ gives a smooth rational curve $\Ga$ of $S$ such that $\Ga^2=1-g$. 
Let $(P_0,Q_0)\in C\times C$ be a point: the map $C\times C\to J(C)\times J(C)$ defined by $(P, Q)\mapsto (P-P_0, Q-Q_0)$ is the Albanese map of $C\times C$ with base point $(P_0,Q_0)$. Composing with the addition map, one obtains a map $C\times C\to J(C)$ that is invariant for the action of $\iota$ and therefore induces a map $a\colon S\to J(C)$, which can be written explicitly as $P+Q\mapsto P+Q-P_0-Q_0$. Using the universal property,  one shows that $a$ is the Albanese map of $S$.  By the Riemann--Roch theorem,  if $C$ is not hyperelliptic $a$ is injective,  while if $C$ is hyperelliptic $a$ contracts $\Ga$ to a point and is injective on $S\setminus \Gamma$. 
Since $H^0(\omega_S)$ is the pull back of $H^0(\Omega^2_{J(C)})=\wedge^2H^0(\omega_C)$ via the Albanese map $a$, the points of $S$ where the differential of $a$ fails to be injective are precisely the base points of $|K_S|$. So, if $C$  is not hyperelliptic then $a$ is an isomorphism of $S$ with its image and if $C$ is hyperelliptic, then $a$ gives an isomorphism of $S\setminus \Ga$  with its image.

Notice that as $g$ goes to infinity, the ratio $K^2_S/\chi(S)$ approaches $8$ from below.
\medskip

(c) \underline{Complete intersections.\/} Let $V$ be an irregular variety of dimension $k+2\ge 3$. For instance, one can take  as $V$  an abelian variety or a product of curves not all rational.  Given  $|D_1|, \dots |D_k|$  free and ample linear systems  on $V$ such that $K_V+D_1+\dots +D_k$ is nef and big, we  take $S=D_1\cap\dots \cap D_k$, with $D_i\in |D_i|$ general, so that $S$ is smooth. By the adjunction formula, $K_S$ is the restriction to $S$ of $K _V+D_1+\dots + D_k$, hence $S$ is minimal of general type. Since the $D_i$ are ample,  the Lefschetz Theorem for hyperplane sections gives an isomorphism $H_1(S,\Z)\iso  H^1(V,\Z)$. Hence the Albanese map of $S$ is just the restriction of the Albanese map of $V$. 

The numerical invariants of $S$ can be computed by means of standard exact sequences on $V$. If $k=1$ and $D_1\in |rH|$, where $H$ is a fixed ample divisor, one has:
$$K^2_S=r^3H^3+O(r^2), \quad \chi(S)=r^3H^3/6+ O(r^2),$$
so the ratio $K^2_S/\chi(S)$ tends to $6$ as $r$ goes to infinity.
Similary, for $k=2$ and $D_1, D_2\in |rH|$, one has:
$$K^2_S=4r^4H^4+O(r^3), \quad \chi(S)=7r^4H^4/12+ O(r^3),$$
and $K^2_S/\chi(S)$ tends to $48/7$ as $r$ goes to infinity.
\medskip

(d) \underline{Double covers.\/} If $Y$ is an irregular surface, any  surface $S$ that dominates $Y$ is irregular, too, and $\albdim(S)\ge \albdim (Y)$. The simplest instance of this situation in which the map $S\to Y$ is not birational is that of a double cover. 
A smooth double cover of a variety $Y$ is determined uniquely by a line bundle $L$ on $Y$ and a smooth divisor $D\in |2L|$.
Set $\mathcal E:=\OO_Y\oplus L\inv$ and let $\Z_2$ act on $\mathcal E$  as multiplication by 1 on  $\OO_Y$ and multiplication by $-1$ on $L\inv$. To define on $\mathcal E$ an $\OO_Y$-algebra structure compatible with this $\Z_2$-action it suffices to give  a map $\mu\colon L^{-2}\to \OO_Y$: we take $\mu$ to be a section whose zero locus is $D$,  set $S:=\Spec E$ and let $\pi\colon S\to Y$ be the natural map. $S$ is easily seen to be smooth iff $D$ is. By construction, one has:
$$H^i( \OO_S)=H^i(\OO_Y)\oplus H^i(L\inv).$$
In particular, if $L$ is  nef and big  then by Kawamata--Viehweg vanishing $H^1(\OO_S)=H^1(\OO_Y)$, hence the induced map  $\Alb(S)\to \Alb(Y)$ is an isogeny.
In fact, it is actually an isomorphism: since $H^0(\Omega^1_S)=H^0(\Omega^1_Y)$, the induced $\Z_2$ action on $\Alb(S)$ is trivial. Since $D=2L$ is nef and big and effective, it is nonempty and therefore we may  choose a base point  $x_0\in S$ that is fixed by $\Z_2$. The Albanese map with base point $x_0$ is $\Z_2$-equivariant, hence it descends to a map $Y\to \Alb(S)$.  So by  the universal property there is a morphism $\Alb(Y)\to \Alb(S)$ which is the inverse of the morphism $\Alb(S)\to \Alb(Y)$ induced by $\pi$. 

A local computation gives the following pull back formula for the canonical divisor:
$$K_S=\pi^*(K_Y+L).$$
By this formula, if $K_Y+L$ is nef and big the surface $S$ is minimal of general type. 
The numerical invariants of $S$ are:
$$K^2_S=2(K_Y+L)^2,\quad  \chi(S)=2\chi(Y)+L(K_Y+L)/2.$$
Hence for ``large'' $L$, the ratio $K^2_S/\chi(S)$ tends to $4$.

If $Y$ is an abelian surface and $L$ is ample, one has $\albdim(S)=2$ and $K^2_S=4\chi(S)$.

\subsection{Examples}\label{ssec:examples}
The constructions of irregular surfaces of  \S \ref{ssec:constructions} can be combined to produce more sophisticated examples, as for instance in Example  \ref{ex:albdeg}. 
However   the computations of the numerical invariants  suggest that in infinite families of examples the ratio $K^2/\chi$ converges, so that it does not seem easy  to fill by these methods large areas of the admissible region for the invariants $K^2,\chi$.

Here we collect some existence results for irregular surfaces.
\medskip

(a) \underline{$\chi=1$.}\\
  As explained in \S \ref{ssec:invariants}, $\chi=1$ is the smallest possible value for a surface  of general type. Since $K^2\le 9\chi$ (\S \ref{ssec:invariants}),  in this case we have $K^2\le 9$, hence   by \eqref{eq:debarre} $q=p_g\le 4$.  To our knowledge, the only known example  of irregular surface of general type with $K^2=9$, $\chi=1$ has been recently constructed by Donald Cartwright (unpublished). It has $q=1$.

We recall briefly what is known about the classification of these surfaces for the possible values of $q$. \begin{itemize} 
\item[q=4:] $S$ is the product of two curves of genus 2 by a result of Beauville (cf. Theorem \ref{thm:2q-4}). So $K^2=8$ in this case. 
\item[q=3:] By \cite{HaconPardini} and \cite{Pirola} (cf. also \cite{CataneseCilibertoMLopes}) these surfaces belong to two families. They are either the symmetric product $S^2C$ of a curve of genus $C$ ($K^2=6$) or free $\Z_2$-quotients of a product of curves $C_1\times C_2$ where $g(C_1)=2$, $g(C_2)=3$ ($K^2=8$).
\item[q=2:] Surfaces with $p_g=q=2$ having an irrational pencil (hence in particular those with $\albdim(S)=1$) are classified in \cite{Zucconi}. They have either $K^2=4$ or $K^2=8$.

Let $(A, \Theta)$ be a principally polarized abelian surface $A$. A double cover $S\to A$ branched on a smooth curve of $|2\Theta|$ is a minimal surface of general type with $K^2=4$, $p_g=q=2$ and it has no irrational pencil iff $A$ is simple (cf. \S \ref{ssec:constructions}, (d)). In  \cite{CilibertoMLopes} it is proven that this is the only surface with $p_g=q=2$ and non birational bicanonical map that has no pencil of curves of genus 2. An example with $p_g=q=2$, $K^2=5$ and no irrational pencil is constructed in  \cite{ChenHacon}. 

\item[q=1:] For  $S$ a  minimal surface of general type with $p_g=q=1$, we denote  by $E$  the Albanese curve of $S$ and by $g$ the genus of the general fiber of the Albanese pencil $a\colon S\to E$.

The case $K^2=2$ is  classified in \cite{Catanese_g2}. These surfaces are constructed as follows. Let $E$ be an elliptic curve with origin $O$.  The map $E\times E\to E$ defined by
$(P,Q)\mapsto P+Q$ descends to a map $S^2E\to E$ whose fibres are smooth rational curves.
We denote by $F$ the algebraic equivalence class of a fibre of $S^2E\to E$. The curves
$\{P\}\times E$ and $E\times \{P\}$ map to curves $D_P\subset S^2E$ such that
$D_PF=D_P^2=1$. The curves $D_P$, $P\in A$, are algebraically equivalent and 
$h^0(D_P)=1$. We denote by $D$ the algebraic equivalence class of $D_P$. Clearly
$D$ and $F$ generate the N\'eron-Severi group  of $S^2E$.
All the surfaces $S$ are (minimal desingularizations of) double covers of $S^2E$ branched on a divisor $B$ numerically
equivalent to $6D-2F$ and with at most simple singularities.
The composite map $S\to S^2E\to E$ is the Albanese pencil of $X$ and its general fibre has genus 2.

The case $K^2=3$ is studied in \cite{CataneseCiliberto91} and \cite{CataneseCiliberto93}. One has either $g=2$ or $g=3$. If $g=2$, then $S$ is birationally a double cover of $S^2E$, while if $g=3$ $S$ is birational to a divisor in $S^3E$. For $K^2=4$, several components of the moduli space are constructed in \cite{Pignatelli} (all these examples have $g=2$). The papers \cite{Rito1}, \cite{Rito2}, \cite{Rito3} contain examples with $K^2=2,\dots 8$. 
The case in which $S$ is birational to a quotient $(C\times F)/G$, where $C$ and $F$  are curves and $G$ is a finite group is considered in \cite{CarnovalePolizzi}, \cite{MistrettaPolizzi}, \cite{Polizzi1}, \cite{Polizzi2}: when $(C\times F)/G$ has at most canonical singularities the surface $S$ has $K^2=8$, but there are also examples with $K^2=2,3,5$.

\end{itemize}
\medskip 

(b) \underline{The line $K^2=2\chi$.}\\
As  pointed out in Proposition \ref{prop:2chi}, for irregular surfaces the lower bound for the ratio $K^2/\chi$ is 2. Irregular surfaces attaining this lower bound were studied in \cite{Horikawa_genus2} and \cite{HorikawaV}. Their structure is fairly simple: they have $q=1$,  the  fibers of the Albanese pencil $a\colon S\to E$  have  genus 2 (cf. also Proposition \ref{prop:2chiq1}) and the quotient of the canonical model of $S$ by the hyperelliptic involution is a $\pp^1$-bundle over $E$.  The moduli space of these surfaces is studied in  \cite{HorikawaV}.

We just show here that for every integer $d>0$ there  exist a minimal irregular surface of general type with $K^2=2\chi$ and $\chi=d$.
In (a) above we have sketched the construction of such a surface $S$ with $K^2_S=2$, $\chi(S)=1$. Let $a\colon S\to E$ be the Albanese pencil and let $E'\to E$ be an unramified cover of degree $d$. Then the map  $S'\to S$ obtained from $E'\to E$ by taking base  change with $S\to E$ is a  connected \'etale cover and $S'$ is minimal of general type with $K^2=2d$, $\chi=d$. By construction $S'$ maps onto $E'$, hence $q(S')>0$. By Proposition \ref{prop:2chiq1} we have $q(S')=1$, hence $S'\to E'$, having connected fibers, is the Albanese pencil of $S'$.

Alternatively, here is a direct construction for $\chi$ even.  Let $Y=\pp^1\times E$, with $E$ an elliptic curve and let $L:=\OO_{\pp^1}(3)\btimes \OO_E(kO)$, where $k\ge 1$ is an integer and $O\in E$ is a point.
Let $D\in |2L|$ be a smooth curve and let $\pi\colon S\to Y$ be the double cover given by the relation $2L\equiv D$ (cf. \S \ref{ssec:constructions}, (d)). The surface is smooth, since $D$ is smooth, and it is minimal of general type since   $K_S=\pi^*(K_Y+L)=\pi^*(\OO_{\pp^1}(1)\btimes\OO_E(kO))$ is ample. The invariants are:
$$K^2=4k, \chi=2k.$$
One shows as above that $q(S)=1$ and $S\to E$ is the Albanese pencil.
\medskip

(c) \underline{Surfaces with an irrational pencil with general fiber of genus $g$.}\\
We use the same construction as in the previous case.  Let $g\ge 2$.  Take $Y=\pp^1\times E$ with $E$ an elliptic curve, $\De$ a divisor of positive degree of $E$ and set $L:=\OO_{\pp^1}(g+1)\btimes \OO_E(\De)$. Let $D\in |2L|$ be a smooth curve and $\pi\colon S\to Y$ the double cover given by the relation $2L\equiv D$. $S$ is smooth minimal of general type, with invariants:
 $$K^2_S=4(g-1)\deg \De, \quad\chi(S)=g\deg \De .$$
 The projection $Y\to E$ lifts to a pencil $S\to E$ of hyperelliptic curves of genus $g$.
 Here $K^2_S/\chi(S)$ is equal to $\frac{4(g-1)}{g}$,  the lowest possible value  (cf. Theorem \ref{thm:slope}).

(d) \underline{The line $K^2=9\chi$.}\\
By \cite[Thm. 2.1]{miyaoka} minimal surfaces of general type with $K^2=9\chi$ have ample canonical class. By Yau's results (\cite{yau}), surfaces with $K^2=9\chi$ and ample canonical class are quotients of the unit ball in $\C^2$ by a  discrete  subgroup. The existence of several examples has been shown using this description (cf. \cite[\S 9, Ch.VII]{bpv}).

Three examples have been constructed in \cite{Hirzebruch_lines}  as Galois covers of the plane branched on an arrangement of lines. 

For later use, we sketch here one of these  constructions. Let $P_1, \dots  P_4\in\pp^2$ be points in general positions and let  $L_1, \dots L_6$ be equations for the  lines through $P_1,\dots P_4$.  Let $X\to\pp^2$ be the normal finite cover corresponding to the  field inclusion $\C(\pp^2)\subset \C(\pp^2)((L_1/L_6)^{\frac1 5},\dots (L_5/L_6)^{\frac1 5})$. The cover $X\to \pp^2$ is abelian with Galois group $\Z_5^5$ and  one can show, for instance by the methods of \cite{ritaabel}, that $X$ is singular over $P_1,\dots P_4$ and that the cover $S\to \hat{\pp^2}$ obtained by blowing up $P_1,\dots P_4$ and taking base change and normalization is smooth.
The cover $S\to \hat \pp^2$ is branched of order 5 on the union $B$ of the exceptional curves and of the strict transforms  of the lines $L_j$. 
Hence the canonical class $K_S$ is numerically equivalent to the pull back of $\frac 1 5(9L-3(E_1+\dots +E_4))$, where $L$ is the pull back on $\hat \pp^2$ class of a line in $\pp^2$ and the $E_i$ are the exceptional curves of $\hat \pp^2\to \pp^2$. It follows that $K_S$  is ample and $K^2_S=3^2\cdot 5^4 $. 
 The divisor $B$ has 15 singular points, that are precisely the points of $\hat \pp^2$ whose preimage consists of $5^3$ points. Hence, denoting  by $e$ the topological Euler characteristic of a variety, we have:
$$c_2(S)=e(S)=5^5[e( \hat{\pp^2})-e(B)] +5^4[e(B)-15]+5^3\cdot 15=3\cdot 5^4.$$
Thus  $S$ satisfies $K^2=3c_2$ or, equivalently, $K^2=9\chi$. In \cite{ishida} it is shown that the irregularity $q(S)$ is equal to 30. To prove that $\albdim S=2$, by the discussion in \S \ref{ssec:pencils} it is enough to show  that $S$ has more than one irrational pencil. The surface $\hat \pp^2$ has 5 pencils of smooth rational curves, induced by the systems $h_i$ of lines through  each  of the $P_i$, $i=1,\dots 4$,  and by the system $h_5$ of conics through $P_1,\dots P_4$. For $i=1,\dots 5$, denote by $f_i\colon S\to B_i$ the pencil induced by $h_i$ and denote by $F_i$ the general fiber of $f_i$.
For $i=1,\dots 5$, the subgroup $H_i<\Z_5^5$ that maps $F_i$ to itself has order $5^3$ and the restricted cover $F_i\to\pp^1$ is branched at 4 points. So  $F_i$ has genus $76$  by the Hurwitz formula.
There is  a commutative diagram:
\begin{equation}\label{diag:Bi}
\begin{CD}
S@>>>\hat \pp^2\\
@V{f_i}VV @VVV\\
B_i@>>>\pp^1.
\end{CD}
\end{equation}
where the map $B_i\to \pp^1$ is an abelian cover  with Galois group $\Z_5^5/H_i\iso \Z_5^2$. The branch points of $B_i\to\pp^1$ correspond to the multiple fibres of $f_i$, hence there are 3 of them and $B_i$ has genus $6$ by the Hurwitz formula. One computes  $F_iF_j=5$ for $i\ne j$, hence the pencils $F_i$ are all distinct. 

Since the group $H_i\cap H_j$ acts faithfully on the set $F_i\cap F_j$ for $F_i$, $F_j$ general, it follows that $H_i\cap H_j$ has order 5 and $H_i+H_j=\Z_5^5$. We use this remark  to show that $H^0(\Omega^1_S)=\oplus_{i=1}^5 V_i$, where $V_i:=f_i^*H^0(\omega_{B_i})$, and therefore that $\Alb(S)$ is isogenous to $J(B_1)\times\dots \times J(B_5)$. Since $q(S)=30$ and $\dim V_i=6$, it is enough to show that the $V_i$ are in direct sum in $H^0(\Omega^1_S)$.  Each subspace $V_i$ decomposes under the action of $\Z_5^5$ as a direct sum of eigenspaces relative to some subset of the group of characters $\Hom(\Z_5^5,\C^*)$. Notice that the trivial character never occurs in the decomposition since $B_i/\Z_5^5=\pp^1$. Diagram \eqref{diag:Bi} shows that the characters occurring in the decomposition of $V_i$ belong to $H_i^{\perp}$. Since $H_i+H_j=\Z_5^5$ for $i\ne j$, it follows that each   character  occurs in the decompositions of  at most one the $V_i$,  and therefore that there is no linear relation among the $V_i$.
 \medskip

(e) \underline{The ratio $K^2/\chi$.}

In  \cite{sommese} Sommese shows that the ratios $K^2(S)/\chi(S)$, for $S$ a minimal surface of general type,  form a dense set in the admissible interval $[2,9]$. His construction can be used to prove:
\begin{prop}\label{prop:density}
\begin{enumerate}
\item The ratios $K^2_S/\chi(S)$, as  $S$ ranges among  surfaces with $\albdim(S)=1$, are dense in the interval $[2,8]$.
\item  The ratios $K^2_S/\chi(S)$, as  $S$ ranges among  surfaces  with $\albdim(S)=2$, are dense in the interval $[4,9]$.
\end{enumerate}
\end{prop}
\begin{proof}
Let $X$ be a minimal surface of general type and let $f\colon X\to B$ be an irrational pencil. Denote by $g\ge 2$ the genus of a general fibre of $f$ and write $K^2$, $\chi$ for $K^2_X$, $\chi(X)$.
 Given positive integers $d, k$, we construct a surface  $S_{d,k}$  as follows:

(1) we  take an unramified degree $d$ cover $B'\to B$ and let $Y_d\to X$ be the cover obtained by taking base change with $f\colon X\to B$,

(2) we take a double cover $B''\to B'$ branched on $2k>0$  general  points and let $S_{d,k}\to  Y_d$ be the cover obtained from $B''\to B'$ by base change. 

The   \'etale cover $Y_d\to X$ is connected, hence  $Y_d$ is a  minimal surface of general type with $K^2_{Y_d}=dK^2$ and $\chi(Y_d)=d\chi$. By   \S \ref{ssec:constructions}, (d), the surface $S_{d, k}$ is smooth, since the branch points of $B''\to B'$ are general, and it is minimal  of general type since $K_{S_{d,k}}$ is numerically the pull back of $K_{Y_d}+kF$, where $F$ is a fibre of $Y_d\to B'$ ($F$ is the same as  the general fiber of $X\to B$).
By the formulae for double covers  we have:
\begin{equation}
\frac{K^2_{S_{d,k}}}{\chi(S_{d,k})}=\frac{2dK^2+8k(g-1)}{2d\chi+k(g-1)}=\frac{K^2}{\chi}+\left(8-\frac{K^2}{\chi}\right)\frac{k(g-1)}{2d\chi+k(g-1)}.
\end{equation}
The formula above shows that the ratio $\frac{K^2_{S_{d,k}}}{\chi(S_{d,k})}$ is in the interval $[8, K^2/\chi]$ if $K^2\ge 8\chi$ and it is in $[K^2/\chi, 8]$ otherwise. It is not difficult to show (cf. \cite{sommese})  that as $d,k$ vary one obtains a dense set in the appropriate interval.

Now to prove the statement it is enough to apply the construction to suitable surfaces. 
If one takes $X$ to be the surface with $K^2=9\chi$ described in (c) and $f\colon X\to B$ one of the 5 irrational pencils of $X$, then the surfaces $S_{d,k}$ have Albanese dimension 2 and the ratios of their  numerical invariants are dense in $[8,9]$.

If one takes $X$ to be a double cover of $E\times E$  branched on a smooth ample curve as in \S \ref{ssec:constructions}, (d)  and  $f\colon X\to E$ one of  the induced pencil, then the surfaces $S_{d,k}$ have Albanese dimension 2 and the ratio of their numerical invariants are dense in $[4,8]$.

Finally, we take $X$ an irregular surface with $K^2=2\chi$. Since   $q(X)=1$  (cf. (c) above),  we can take  $f\colon X\to B$ to be  the Albanese pencil. In this case the ratios of the numerical  invariants  of the  surfaces $S_{d,k}$ are dense in the interval $[2,8]$. To complete the proof we show that the surfaces $S_{d,k}$ have Albanese dimension 1. The surfaces $Y_d$ satisfy $K^2=2\chi$, hence they also have $q=1$. The induced pencil $S_{d,k}\to B''$ has genus $k+1$, so we need to show that the irregularity of $S_{d,k}$ is equal to $k+1$. Denote by $L$ the line bundle of $Y_d$ associated to the double cover $S_{d,k}\to Y_d$. By construction $L=\OO_{Y_d}(F_1+\dots+F_k)$, where  the $F_i$ are fibers of the Albanese pencil,  and if $k>1$ we can take the $F_i$ smooth and distinct. 
We have $q(S_{d,k})=q(Y_d)+h^1(L\inv)=1+h^1(L\inv)$. Finally, $h^1(L\inv)=k$ can be proven using the restriction sequence:
$$0\to L\inv \to\OO_{Y_d}\to\OO_{F_1+\dots+F_k}\to 0.$$ 
(Notice that the map $H^1(\OO_{Y_d})\to H^1(\OO_{F_1+\dots+F_k})$ is 0, since the curves  $F_i$ are contracted by the Albanese map). 
\end{proof} 
\begin{remark}\label{rem:pencil}
 Due to the method of proof, all the surfaces constructed in the proof of Proposition \ref{prop:density} have an irrational pencil. The examples in \S \ref{ssec:constructions}  show that, for instance,  $4, 6, \frac{48}{7}, 8$ are accumulation points for the ratio $K^2/\chi$ of irregular surfaces without irrational pencils.  We have no further  information on the distribution of the ratios $K^2/\chi$ for these surfaces.

\end{remark}
\section{The Castelnuovo--De Franchis inequality}\label{sec:CDF}

 Let $S$ be an irregular minimal surface of general type. Set $V:=H^0(\Omega^1_S)$ and denote by  $w\colon \wedge^2 V \to H^0(\omega_S)$ the natural map.
 If $f\colon S\to B$ is a pencil of genus $b\ge 2$,  then $f^*H^0(\Omega^1_B)$ is a subspace of $V$ such that $\wedge^2f^*H^0(\Omega^1_B)$ is contained in $\ker w$ (cf. \S \ref{ssec:pencils}).  Conversely if $p_g=h^0(\omega_S)<2q-3$,  the intersection of  $\ker w$ with the cone of decomposable elements is non zero and by  \ref{thm:CDF}, $S$ has  a pencil of genus $b\ge 2$.  Thus:

  \begin{thm}{\rm  (Castelnuovo-De Franchis inequality)}
  Let $S$ be an irregular  surface of general type having no irrational pencil of genus $b\ge 2$. Then $p_g\geq 2q-3$. 
  \end{thm}
  
 In fact, using this theorem and positivity properties of the relative canonical bundle of a fibration (cf. Theorem \ref{thm:arakelov}),  Beauville showed:
 \begin{thm}[\cite{Beauville_appendix}]\label{thm:2q-4}
 Let $S$ be a minimal surface of general type. Then $p_g\geq 2q-4$ and if equality holds then $S$ is the product of a  curve of genus $2$ and a curve of genus $q-2\geq 2$.
 \end{thm}
So  surfaces satisfying $p_g=2q-4$ have a particularly simple structure and it is natural to ask what are the  irregular surfaces satisfying $p_g=2q-3$.  Those having an irrational pencil have again a simple structure, as explained in the following theorem, which  was proven in \cite{CataneseCilibertoMLopes} for $q=3$, in \cite{bnp} for $q=4$ and finally for $q\ge 5$ in \cite{RitaMarg_advances}.

\begin{thm}\label{irrpencil} 

 Let $S$ be a minimal surface of general type satisfying $p_g=2q-3$. If $S$ has an irrational pencil of genus $b\ge 2$, then
there are the following possibilities for $S$:
\begin{enumerate}
\item $S=(C\times F)/\Z_2$, where $C$ and $F$ are genus 3 curves with a free involution ($q=4$);
\item $S$ is the product of two curves of genus 3 ($q=6$);
\item $S=(C\times F)/\Z_2$, where $C$  is a curve of  genus $2q-3$ with a free action of $\Z_2$, $F$  is a curve of genus  2 with a $\Z_2$-action such that $F/\Z_2$ has genus 1 and $\Z_2$ acts diagonally on $C\times F$ ($q\ge 3$). 

\end{enumerate}
In particular $K_S^2=8\chi$.

  \end{thm}
  
  So the main issue  is  to study the case when $S$ has no irrational pencil. Various numerical restrictions on the invariants have been obtained.  For instance if $q\ge 5$, then (cf. (\cite{RitaMarg_advances}):
  \begin{itemize}
   \item  $K_S^2\geq 7\chi(S)-1$;
 \item if $K_S^2<8\chi(S)$, then  $|K_S|$ has fixed components and the degree of the canonical map is 1 or 3;
 \item  if $q(S)\geq 7$  and  $K^2<8\chi(S)-6$, then the canonical map is birational.
 \end{itemize}

 However, it is  hard to say whether these results are sharp, since the only known example of surface with $p_g=2q-3$ and no irrational pencil is the symmetric product of a general curve of genus 3. For low values of $q$ we have the following:
  \begin{itemize}
  \item if $q=3$, then $S$ is the symmetric product of a curve of genus 3 (cf. \S \ref{ssec:examples}, (a));
  \item if $q=4$, then $K^2=16,17$ (cf. \cite{bnp}, \cite{PirolaCausin});
  \item if $q=5$, there exists no such surface (\cite{Pirola_q5}). 
  \end{itemize}

\begin{remark} Recently  Theorem  \ref{thm:CDF} has been generalized  by Pareschi and Popa  to the case of Ka\"hler manifolds of arbitrary dimension:
\begin{thm}[\cite{PareschiPopa}]
Let $X$ be a compact K\"{a}hler manifold  with  $\dim X=\albdim X=n$. If there exists no surjective morphism    $X\to Z$ with $Z$ a normal analytic  variety such that $0<\dim Z=\albdim Z<min\{n, q(Z)\}$, then:
$$\chi(\omega_X)\ge q(X)-n.$$
\end{thm}
\end{remark}

\section{The slope inequality}\label{sec:slope}
\subsection{Relative canonical class and  slope of a fibration}\label{ssec:slope_definitions}
In this section we consider a fibration (``pencil'')  $f\colon S\to B$ with $S$ a smooth projectve surface and  $B$ a smooth curve of genus $b\ge 0$. Recall that a fibration is {\em smooth} iff all its fibers are smooth, and it is  {\em isotrivial} iff all the smooth fibers of $f$ are isomorphic, or, equivalently, if  the fibers over a non empty open set of $B$ are isomorphic. Isotrivial fibrations are also said to have  ``constant moduli''.

We assume that the general fiber $F$ of $f$ has genus $g\ge 2$ and that $f$ is {\em relatively minimal}, namely that there is no $-1$-curve contained in the fibers of $f$. Notice that  these assumptions are always satisfied if $S$ is minimal of general type.
Notice also that given a non minimal fibration $f$ it is always possible to pass to a minimal one   by blowing down the $-1$-curves in the fibers 

The {\em relative canonical  class}  is defined by $K_f:=K_S-f^*K_B$. We also write $\omega_f$ for the corresponding line bundle $\OO_S(K_f)=\omega_S\otimes f^*\omega_B\inv$. $K_f$ has the following positivity properties:
\begin{thm}[Arakelov, cf. \cite{Beauville_appendix}]\label{thm:arakelov}
 Let $f$ a relatively minimal fibration of genus $g\ge 2$. Then:
\begin{enumerate}
\item $K_f$ is nef, hence $K^2_f=K^2_S-8(g-1)(b-1)\ge 0$;
\item if $f$ is not isotrivial, then:
\begin{itemize}
\item[\rm(a)] $K^2_f>0$
\item[\rm(b)] $K_fC=0$ for an irreducible curve $C$ of $S$ if and only if $C$ is a $-2$-curve contained in a fiber of $f$.
\end{itemize}
\end{enumerate}
 \end{thm}

Let $f\colon S\to B$ be relatively minimal and let $\ol S$ be the surface obtained by contracting the $-2$ curves contained in the fibers of $f$. There is an induced fibration $\ol f\colon \ol S\to  B$  and, since ${\ol S}$ has canonical singularities, $K_f$ is the pull back of  $K_{\ol f}:= K_{\ol S}-\ol f^*K_B$. By the Nakai criterion for ampleness,   Theorem \ref{thm:arakelov}, (ii) can be restated by saying that if  $\ol f$ (equivalently, $f$)  is not isotrivial then $K_{\ol f}$ is ample on $\bar{S}$.

Given a pencil $f$ with general fiber of genus $g$, the push forward $f_*\omega_f$ is a rank $g$ vector bundle on  $B$ of degree $\chi_f=\chi(S)-(b-1)(g-1)$.  Recall that a vector bundle $E$ on a variety  $X$ is said to be {\em nef } if the tautological line bundle $\pp(E)$ is nef.
\begin{thm}[Fujita, cf. \cite{Fujita_kahler}, \cite{Beauville_appendix}]\label{thm:fujita}
Let $f\colon S\to B$ be a relatively minimal fibration with general fiber of genus $g\ge 2$.
Then:
\begin{enumerate}
\item $f_*\omega_f$ is nef. In particular, ${\OO_{\pp(f_*\omega_f)}(1)}\,^g=\chi_f \ge0$;
\item $\chi_f=0$ if and only if $f$ is smooth and isotrivial.
\end{enumerate}
\end{thm}

From now one we assume that the relatively minimal fibration $f\colon S\to B$ is not isotrivial. The {\em slope} of $f$ is defined as:
\begin{equation}
\la(f):=\frac {K^2_f}{\chi_f}=\frac{K^2_S-8(b-1)(g-1)}{\chi(S)-(b-1)(g-1)}.
\end{equation}
By Theorem \ref{thm:arakelov} and Theorem \ref{thm:fujita}, $\lambda(f)$ is well defined and $>0$.

The main result concerning the slope is:
\begin{thm}[Slope inequality]\label{thm:slope}
Let $f\colon S\to B$ be a relatively minimal fibration with fibers of genus $g\ge 2$. If $f$ is not smooth and isotrivial, then:
$$\frac{4(g-1)}{g}\le \la(f)\le 12.$$
\end{thm}
\smallskip

The inequality $\la(f)\le 12$ follows from Noether's formula  $12\chi(S)=K^2_S+c_2(S)$ and from the following  well known formula for the Euler characteristic of a fibered surface (cf. \cite[Prop. II.11.24]{bpv}):
\begin{equation}\label{eq:c2}
c_2(S)=4(b-1)(g-1)+\sum_{P\in T} e(F_P)-e(F),
\end{equation}
where $e$ denotes the topological Euler characteristic, $T$ is the set of critical values of $f$, $F_p$ is the fiber of $f$ over  the point $P$ and $F$ is a general fiber. For any point $P\in B$  one has $e(F_P)\ge e(F)$, with equality holding only if $F_P$ is smooth (cf. {\em ibidem}). Hence, the main content  of Theorem \ref{thm:slope} is the lower bound $\la(f)\ge \frac{4(g-1)}{g}$.

We have seen  (Proposition \ref{prop:density} and Remark \ref{rem:pencil}) that the ratios $K^2/\chi$ for surfaces with an irrational pencil are dense in the interval $[2,9]$.  The slope inequality gives a lower bound for this ratio in terms of the genus $g$ of the general fiber of the pencil.

 \begin{prop}\label{prop:2chiq1} Let $S$ be a minimal surface of general type that has  an irrational pencil  $f\colon S\to B$ with general fiber of genus $g$.
Then:
 $$K^2_S\ge \frac{4(g-1)}{g}\chi(S)\ge 2\chi(S).$$
In particular, if $K^2_S=2\chi(S)$, then $g=2$ and $B$ has genus 1.
\end{prop}
\begin{proof} Assume that the pencil $S$ is not smooth and isotrivial. If   $K^2_S\ge 8\chi(S)$ then of course the statement holds. If $K^2<8\chi$, then  $K^2_S/ \chi(S)\ge \lambda(f)$ and the stament follows by the slope inequality. 

If $f$ is smooth and isotrivial, then by \cite[\S 1]{Serrano}, $S$ is a quotient $(C\times D)/G$ where $C$ and $D$ are curves of genus $\ge 2$ and $G$ is a finite group that acts freely. In particular, $K^2_S=8\chi(S)$ and the inequality is satisfied also in this case.

If $K^2_S=2\chi(S)$, then by the previous remarks $f$ is not smooth and isotrivial. We have:
$$2= \frac{K^2_S} {\chi(S)}\ge \lambda(f)\ge\frac{4(g-1)}{g}.$$
It follows immediately that $g=2$ and $B$ has genus 1. 
\end{proof}

 Further  applications of the slope inequality are given in \S \ref{sec:severi}.

\subsection{History and proofs}\label{ssec:slope_proofs}

Theorem \ref{thm:slope} has been stated and proven first in the case of hyperelliptic fibrations in \cite{Persson}, \cite{HorikawaV}. 

The general statement was  proven in \cite{Xiao_slope} and, independently, in \cite{CornalbaHarris} under the extra assumption that the fibers of $f$ be semistable curves, i.e. nodal curves.  The  proof  of \cite{CornalbaHarris} has been recently generalized  in \cite{Stoppino} to cover the general case.  Yet another proof has been  given  in \cite{Moriwaki}.

In addition to the original papers, one can find  in \cite{Azumino} a nice account of Xiao's proof and of Moriwaki's proof, and in \cite{Stoppino} a generalization of Cornalba-Harris proof. Hence  it seems superfluous to include the various proofs here.

We just wish to point out that the three methods of proof are different.  Xiao's proof uses the Harder-Narasimhan filtration  of the vector bundle $f_*\omega_f$  and Clifford's Lemma. 

Cornalba-Harris proof uses GIT and relies on the fact that a canonically embedded curve of genus $g\ge 3$ is stable (fibrations whose   general fiber is hyperelliptic are treated separately). 

Moriwaki's proof is based on Bogomolov's instability theorem for vector bundles on surfaces.

\subsection{Refinements and generalizations}\label{ssec:slope_remarks}\ 

(a) \underline{Fibrations attaining the lower or the upper bound for the slope.}\\
The examples constructed in   \S \ref{ssec:examples}, (c) show that the lower bound for $\la(f)$ given in  Theorem \ref{thm:slope} is sharp. By construction, the general fiber in all these examples is hyperelliptic. This is not a coincidence: in \cite{Konno_SNS} it is proven that the general fiber of a  fibration attaining the minimum possible value of the slope is  hyperelliptic.
On the other hand, if $\lambda(f)=12$, then by Noether's formula one has $c_2(S)=4(g-1)(b-1)$. Hence by \eqref{eq:c2}, this happens if and only if all the fibers of $f$ are smooth.
\smallskip

(b) \underline{Non hyperelliptic fibrations.}\\
Since, as explained in (a), the minimum value of the slope is attained only by hyperelliptic fibrations, it is natural to look  for  a better bound for non hypereliptic fibrations. In \cite{Konno_SNS}, such a lower bound is established for $g\le 5$.
In \cite{Konno_Clifford}, it is shown that the inequality
 $$\lambda(f)\ge \frac{6(g-1)}{g+1}$$
 holds if: (1) $g$ is odd, (2) the general fiber of $f$ has maximal Clifford index, (3) Green's conjecture is true for curves of genus $g$.

Since in fact Green's conjecture for curves of odd genus and maximal Clifford index  has later been proven(\cite{Voisin} and \cite{Ramanan}), Konno's result actually holds under assumption (1) and (2).

The influence of the Clifford index and of the {\em relative irregularity} $q_f:=q(S)-b$ has been studied also in \cite{BarjaStoppino}.
\smallskip

(c) \underline{Fibrations with general fiber of special type.}\\ Refinements of the slope inequality have been obtained also under the assumption that the general fiber has some special geometrical property. 

In \cite{Konno_Trigonal} it is shown that if the general fiber of $f$ is trigonal and $g\ge 6$, then $\lambda(f)\ge \frac{14(g-1)}{3g+1}$.
In \cite{BarjaStoppino2} it is shown that the better  bound $\lambda(f)\ge \frac{5g-6}{g}$ holds if $g\ge 6$ is even and the general fiber of $f$ has Maroni invariant $0$.

The case in which the general  fiber of $f$  has an involution with quotient a curve of genus $\ga$  has been considered in several papers (\cite{Barja_bielliptic}, \cite{BarjaZucconi}, \cite{CornalbaStoppino}): the most general result in this direction is the inequality 
$$\lambda(f)\ge \frac{4(g-1)}{g-\ga},$$
 which holds for  $g>4\ga+1$ (\cite{CornalbaStoppino}).
\smallskip

(d) \underline{Generalizations to higher dimensions.}\\
Let $f\colon X\to B$ be a fibration, where $X$ is an $n$-dimensional $\Q$-Gorenstein variety and $B$ is a smooth  curve. As in the case of surfaces, one can consider the relative canonical divisor $K_f:=K_X-f^*K_B$ and define the slope of $f$ as:
$$\lambda(f):=\frac{{K_f}^n}{\deg f_*(\OO_X(K_f))}.$$
This situation is studied in \cite{BarjaStoppino2}, where some lower bounds are obtained  under quite restrictive assumptions on the fibers. 

The relative numerical invariants of threefolds fibered over a curve have  also been studied in \cite{Ohno} and \cite{Barja_threefolds}.
\section{The Severi inequality}\label{sec:severi}
\subsection{History and proofs}\label{ssec:Severi_proofs}

In  a paper of 1932 (\cite{Severi})  Severi stated the following:  
\begin{thm}[Severi inequality]\label{thm:Severi}

Let $S$  be a minimal surface of general type with  $\albdim(S)=2$, then:
$$K^2_S\ge 4\chi(S).$$
\end{thm}

Unfortunately, Severi's proof contained a fatal gap,  that was pointed out by Catanese  (\cite{Catanese_Ravello}), who posed the inequality as a conjecture.  More or less at the same time, Reid (\cite{Miles_pi1}) made the following conjecture, that for irregular surfaces is a consequence  of Theorem \ref{thm:Severi} (cf. Proposition  \ref{prop:Reid}). 
\begin{conj}[Reid]\label{conj:Reid}
If $S$ is a minimal surface of general type such that $K_S^2<4 \chi(S)$ then either $\pionealg(S)$ is finite or  there exists a finite \'etale cover $S'\to S$ and a pencil $f\colon S'\to B$ such that the induced  surjective map on the algebraic  fundamental groups has finite kernel.
\end{conj}
(Recall that for a complex variety $X$ the algebraic fundamental  group $\pionealg(X)$ is the profinite completion of  the topological fundamental group $\pi_1(X)$). 

Motivated
by this conjecture, Xiao wrote the foundational paper \cite{Xiao_slope} on  surfaces fibred over a curve,
in which he proved both  the Severi inequality  in the special case of a surface with an irrational pencil and the slope inequality (cf. \S \ref{sec:slope}).  

Building on the  results of \cite{Xiao_slope}, Severi's conjecture was proven by Konno (\cite{Konno_Severi}) for even surfaces, namely surfaces such that the canonical class is divisible by 2 in the Picard group.
At the end of the 1990's, the conjecture was almost solved by
Manetti (\cite{Manetti_Severi}),  who proved  it  under  the additional assumption that
the surface have ample canonical bundle. Finally,  the inequality was proven in \cite{Rita_Severi}.

The proof given in \cite{Rita_Severi} and 
the proof  given  in \cite{Manetti_Severi} for $K$ ample are completely different. We sketch  briefly both proofs: 
\medskip

(a) \underline{Proof for $K$ ample (\cite{Manetti_Severi}):}
Let  $\pi\colon  \pp(\Omega^1_S)\to S$ be the projection and let $L$ be the hyperplane class of $\pp(\Omega^1_S)$.  Assume for simplicity that $H^0(\Omega^1_S)$ has no base curve. Then two general elements $L_1, L_2 \in |L|$ intersect properly, hence $L^2$ is represented by the effective $1$-cycle $L_1\cap L_2$.
One has:
\begin{equation}\label{eq:manetti}
L^2(L+\pi^*K_S)=3(K^2_S-4\chi(S)).
\end{equation}
If $L+\pi^*K_S$ is nef, then Theorem \ref{thm:Severi} follows immediately by \eqref{eq:manetti}. However, this is not the case in general, and one is forced to analyze the cycle $L_1\cap L_2$ more closely.  One can write:
$$L_1\cap L_2=V+\Ga_0+\Ga_1+\Ga_2,$$
where  $V$ is a sum of fibers of $\pi$, and the  $\Ga_i$ are sums of sections of $\pi$. More precisely, the support of  $\pi(\Ga_0)$ is  the union of the curves contracted by the Albanese map $a$,  the support of $\pi(\Ga_1)$ consists  of  the curves not contracted by $a$ but on which the differential of $a$ has rank 1 and $\pi(\Ga_2)$ is supported on curves on which the differential of $a$ is generically non singular.  The term $\Ga_0(L+\pi^*K_S)$ can be $<0$, but by means of a very careful analysis of the components of $\Ga_0$ and of the multiplicities in the vertical cycle $V$ one can show that $L^2(L+\pi^*K_S)=(L_1\cap L_2)(L+\pi^*K_S)\ge 0$. Unfortunately this kind of analysis does not work when $\pi(\Ga_0)$ contains $-2$-curves, hence one has to assume $K_S$ ample.
\medskip

(a) \underline{Proof (\cite{Rita_Severi}):}

Set $A:=\Alb(S)$ and fix a very ample divisor $D$ on $A$. For every integer $d$ one constructs a fibered surface $f_d\colon Y_d\to \pp^1$ as follows:
\smallskip

(1) consider the following cartesian diagram, where $\mu_d$ denotes  the multiplication by $d$: 
\begin{equation}
\begin{CD}
S_d@>{p_d}>> S\\
@V{a_d}VV @VVaV\\
A@>\mu_d>>A
\end{CD}
\end{equation}
and let $H_d:=a_d^*D$. The map $p_d$ is a finite  connected \'etale cover, in particular $S_d$ is minimal of general type.
\smallskip

(2) choose a general pencil $\Lambda_d\subset |H_d|$ and let $f_d\colon Y_d\to \pp^1$ be the relatively minimal fibration obtained by resolving the indeterminacy of the rational map $S_d\to \pp^1$ defined by $\Lambda_d$.
\smallskip

The key observation  of  this  proof is that as $d$ goes to infinity, the intersection numbers $H^2_d$ and $K_{S_d}H_d$ grow slower than  $K^2_{S_d}$ and $\chi(S_d)$. 
As a consequence,  the slope of $f_d$  
converges to the ratio
$K^2_S/\chi(S)$ as $d$ goes to infinity. Since  $g_d$ goes to infinity with $d$,  the Severi 
inequality can be obtained  by applying the   slope inequality (Theorem \ref{thm:slope}) to the fibrations 
$f_d$ and taking the limit for $d\to \infty$.

\subsection{Remarks, refinements and open questions}\label{ssec:remarks}\ \par

 (a) \underline{Chern numbers of surfaces with $\albdim=2$.}\\
 The Severi inequality is sharp, since double covers of an abelian surface branched on a smooth ample curve satisfy $K^2=4\chi$ (cf. \S \ref{ssec:examples}, (d)). Actually, in \cite{Manetti_Severi} it is proven that these are the only surfaces with $K$ ample, $\albdim=2$ and $K^2=4\chi$.  Hence it is natural to conjecture that the canonical models of surfaces with $\albdim=2$ and $K^2=4\chi$ are double covers of abelian surfaces branched on an ample curve with  at most simple singularities.
 
 In addition, by Proposition \ref{prop:density}, the ratios $K^2_S/\chi(S)$ for $S$ a minimal surface  with $\albdim S=2$ are dense in all the admissible interval $[4,9]$.
 \smallskip

(b) \underline{Refinements of the inequality.} 

Assume that the differential of the Albanese map $a\colon S\to A$ is non singular outside a finite set. Then the cotangent bundle $\Omega^1_S$ is nef and $L^3=2(K^2_S-6\chi(S))\ge 0$. This remark suggests that one may expect an inequality stronger than Theorem \ref{thm:Severi} to hold under some  assumption on $a$, e.g., that $a$ be birational.
A  possible way of obtaining a result of this type would be to apply in the  proof of \cite{Rita_Severi}  one of the refined  versions  of the slope inequality (cf. \S \ref{ssec:slope_remarks}, (b) and (c)). Unfortunately, one has very little control on the general fiber of the fibrations $f_d\colon Y_d\to \pp^1$ constructed in the proof.

Analogously, by the result of Manetti mentioned in (a), it is natural to expect that a better bound holds for surfaces with $q>2$ (cf. \cite[\S 7]{Manetti_Severi} for a series of conjectures). A step in this direction is the following:
\begin{thm}[\cite{RitaMarg_Severi}]\label{thm:RM_Severi}
Let $S$ be a smooth surface of maximal Albanese dimension  and irregularity  $q\ge 5$ with $K_S$ ample. Then:
 $$K^2_S\ge 4\chi(S)+\frac{10}{3}q-8.$$
Furthermore,   if $S$ has no irrational pencil and  the Albanese map $a\colon S\to A$ is unramified in codimension 1, then 
$$K^2_S\ge 6\chi(S)+2q-8.$$
\end{thm}

Theorem \ref{thm:RM_Severi} is proven by   using geometrical arguments to give a lower bound  for  the term $L\Ga_2$ in the proof of the Se veri inequality for $K$ ample given in \cite{Manetti_Severi} (cf. \S \ref{ssec:Severi_proofs}, (a)). This is  why one needs to assume $K$ ample. In \cite{RitaMarg_Severi}, it is also  shown that the bounds of Theorem \ref{thm:RM_Severi} can be sharpened if one assumes that the Albanese map or the canonical map of $S$  is not birational.
\smallskip

(c) \underline{Generalizations to higher dimensions.}

The proof of Theorem \ref{thm:Severi} given in \cite{Rita_Severi} (cf. \S \ref{ssec:Severi_proofs}) would work in arbitrary dimension $n$ if one had a slope inequality for varieties of dimension $n-1$. For instance, using the results of \cite{Barja_threefolds}, one can prove:

\begin{thm}\label{thm:Severi_3folds}
Let $X$ be a smooth projective threefold such that $K_X$ is nef and big and $\albdim X=3$. Then:
\begin{enumerate}
\item $K^3_X\ge 4\chi(\omega_X)$;
\item if $\Alb(X)$ is simple,  then $K^3_X\ge 9\chi(\omega_X)$.
\end{enumerate}
\end{thm} 
\begin{proof} We may assume $\chi(\omega_X)>0$.
(Recall $\chi(\omega_X)\ge 0$ by 
\cite[Cor. to Thm.1]{GreenLazarsfeld1}). 
Consider  a fibered threefold $f\colon Y\to B$, where $B$ is a smooth curve, denote by $F$ a general fiber and assume $\chi_f:=\chi(\omega_X)-\chi(\omega_B)(\chi(\omega_F)\ne 0$. Following \cite{Barja_threefolds}  we define:
  $$\lambda_2(f):=\frac{(K_X-f^*K_B)^3}{\chi_f}.$$

  We apply the same construction as in \S \ref{ssec:Severi_proofs}, (b) to get for every positive integer $d$ a smooth fibered threefold $f_d\colon Y_d \to\pp^1$ such that  $\lambda_2(f_d)$ is defined for $d>>0$ and converges to $\frac{K^3_X}{\chi(\omega_X)}$ for $d\to \infty$. Statement (i) now follows by applying  \cite[Thm. 3.1,(i)]{Barja_threefolds} to $f_d$ and taking the limit for $d\to\infty$.
  
  Statement (ii) requires a little more care. The threefold $Y_d$ is the blow up along a smooth curve of an \'etale cover $Z_d\to X$. Since $A:=\Alb(X)$ is simple, $V^1(X):=\{P\in \Pic^0(X)|h^1(-P)>0\}$ is a finite set by the Generic Vanishing theorem of \cite{GreenLazarsfeld1}. Then there are infinite values of $d$ such that  $dP\ne 0$ for every $P\in V^1(X)$.   For those values   $q(Y_d)=q(Z_d)=q(X)$ and $A:=\Alb(X)$ and $\Alb(Y_d)=\Alb(Z_d)$ are isogenous, in particular $\Alb(Z_d)$ is simple. By construction, the general fiber  $F_d$ of $f_d$ is a surface of  maximal Albanese dimension. In addition, since $F_d$ is isomorphic to an element of the nef and big linear system $|H_d|$ (notation as in \S \ref{ssec:Severi_proofs}, (b)), it follows that $q(F_d)=q(X)$ and $\Alb(F_d)$ is isogenous to the simple abelian variety  $\Alb(Y_d)$. Hence $F_d$ has no irrational pencil and we get statement (ii) by applying  \cite[Thm. 3.1,(ii)]{Barja_threefolds} and taking the limit for $d\to\infty$.
  \end{proof}
\begin{remark} In order to keep the proof of Theorem \ref{thm:Severi_3folds} simple, we have made stronger assumptions than necessary: for instance one can assume that $X$ has terminal singularities and, with some more work, one can eliminate the assumption that $\Alb(X)$ be simple in (ii).

\end{remark}

\subsection{Applications}

  We  use   the following result due to Xiao Gang:
\begin{prop}[\cite{Xiao_slope}] \label{prop:Xiao_pi1}
Let $f\colon S\to B$ be a relatively minimal fibration with fibers of genus $g\ge 2$. If $\lambda(f)<4$ and $f$ has no multiple fibers, then there is an exact sequence:
$$1\to N\to \pionealg(S)\to \pionealg(B)\to 1,$$
where $|N|\le 2$.
\end{prop}
We also need the following consequence of the Severi inequality and of the slope inequality.
\begin{lem} \label{lem:consequence}
Let $S$ be  a minimal regular surface of general type $S$ with $K^2_S<4\chi(S)$ that has   an irregular  \'etale cover $S'\to S$.
 Then $S$ has a pencil 
$f\colon S\to\pp^1$ such that:
\begin{enumerate}
\item $f$ has multiple fibers $m_1F_1,\dots m_kF_k$  with    $\sum_j \frac{m_j-1}{m_j}\ge 2$;
\item the general fiber of $f$ has genus $g\le 1+\frac{K^2_S}{4\chi(S)-K^2_S}$.
\end{enumerate}
\end{lem}
\begin{proof}
Up to passing to the Galois closure, we may assume that $S'\to S$ is  a Galois cover with Galois group $G$. 
By the Severi inequality (Theorem \ref{thm:Severi}) the Albanese map of $S'$ is a pencil $a'\colon S'\to B$, where $B$ has genus $b>0$. Clearly $G$ acts on $f$  and on $B$, inducing a  pencil $f\colon S\to B/G$. Since $S$ is regular, $B/G$ is isomorphic to $\pp^1$.  Denote by $\ol G$ the quotient of $G$ that acts effectively on $B$.  Let $y\in B$ be a ramification point of order $\nu$ of the map $B\to B/\ol G=\pp^1$ and let $H<\ol G$ be the stabilizer  of  $y$. The group $H$ is cyclic of order $\nu$ and it acts freely on the fiber   $F_y$ of $a'$ over $y$. It follows that the fiber of $f$ over the image $x$  of $y$ is  a multiple fiber of multiplicity divisible by $\nu$. Let $x_1,\dots x_r\in \pp^1$ be the critical values of $B\to B/\ol G$, let $\nu_i$ be the ramification order of $x_i$ and let $m_1F_1,\dots m_kF_k$ be the multiple fibers of $f$.
Then by the Hurwitz formula  we have:
$$\sum _j\frac{m_j-1}{m_j}\ge \sum_j\frac{\nu_i-1}{\nu_i}=\frac{2b-2}{|\ol G|}+2\ge 2,$$
hence (i) is proven.

To prove (ii), we observe that the fibers of $f$ are quotients (possibly by a trivial action) of the fibers of $a$, hence $g\le \ga$, where $\ga$ is the genus of a general fiber of $a$. In turn,  by the slope inequality one has:
$$\frac{K^2_S}{\chi(S)}\ge \lambda(a)\ge \frac{4(\ga-1)}{\ga},$$
which gives  the required bound  
$$g\le \ga\le 1+\frac{K^2_S}{4\chi(S)-K^2_S}.$$
\end{proof}

(a) \underline{Reid's conjecture for irregular surfaces.}

The Severi inequality implies  Reid's conjecture \ref{conj:Reid}  for irregular surfaces and, more generally,  surfaces that have an irregular \'etale cover:
\begin{prop}\label{prop:Reid}
Let $S$ be a minimal surface of general type with $K^2_S<4\chi(S)$. If $S$ has an irregular cover, then there exists a finite \'etale cover $S'\to S$ and a pencil $f\colon S' \to B$ that induces an exact sequence:
$$0\to N\to \pionealg(S')\to  \pionealg(B)\to 0,$$
where $|N|\le 2$.
\end{prop}
\begin{proof}
Let $X\to S$ be an irregular  \'etale cover. By Theorem \ref{thm:Severi} the Albanese map of $X$ is a pencil $a\colon X\to C$ and,  if $S'\to X$ is \'etale, then the Albanese pencil of $S'$ is obtained  by pulling back the Albanese pencil of $X$ and  taking the Stein factorization.  So, up to taking a suitable base change $B\to C$ and normalizing, we can pass to an \'etale cover $S'\to S$ such that the Albanese pencil $a'\colon S'\to B$ has no multiple fiber. The statement now follows by applying Proposition \ref{prop:Xiao_pi1} to $a'$.
\end{proof}
\begin{remark} By Proposition \ref{prop:Reid}, to prove Reid's conjecture it is enough to show:

{\em   If $S$ is a surface with $K^2_S<4\chi(S)$ that has no irregular \'etale cover,  then $\pionealg(S)$ is finite.}

By the work of several authors, this is known to be true for $K^2<3\chi$ (cf. \cite{RitaMarg_JDG} and  references therein). Moreover in \cite{RitaMarg_JDG} and \cite{CiroRitaMarg} the following sharp  bound on the order of $\pionealg(S)$ is given:

{\em If $K^2_S<3\chi(S)$ and $S$ has no irregular \'etale cover, then $|\pionealg(S)|\le 9$. Furthermore, if  $|\pionealg(S)|=8$ or $9$ then $K^2_S=2$, $p_g(S)=0$ (namely, $S$ is a  Campedelli surface).}

Surfaces with $K^2=2$, $p_g=0$ and  $|\pionealg|=8$, 9 are classified in \cite{MilesRitaMarg}, respectively  \cite{RitaMarg_JEMS}.

However, all the above mentioned results make essential use of Ca\-stel\-nuo\-vo's inequality $K^2\ge 3\chi-10$ for surfaces  whose canonical map is not 2:1 onto a ruled surface. Hence, different methods are needed to deal with surfaces  with  $3\chi\le K^2<4\chi$.

\end{remark}

(b) \underline{Surfaces with ``small'' $K^2$.}

By Proposition \ref{prop:2chi}, a minimal irregular  surface of general type satisfies  $K^2\ge 2\chi$.
Irregular surfaces on or near  the line $K^2=2\chi$ have been  studied by several authors ( \cite{bombieri}, \cite{HorikawaV}, \cite{Miles_pi1}, \cite{Xiao_slope}, \cite{Xiao_hyperell}). As an  application of the Severi inequality and of the slope inequality, we give quick proofs of  some of these results.

\begin{prop}\label{prop:2chi_char}
Let $S$ be a minimal irregular surface of general type. Then:
\begin{enumerate}
\item  If $K^2_S=2\chi(S)$, then $q(S)=1$
 and the fibers of the Albanese pencil $a\colon S\to B$ have genus 2 (cf. Proposition \ref{prop:2chiq1});
\item if $K^2_S<\frac 8 3 \chi(S)$, then the Albanese map is a pencil of curves of genus 2;
\item if $K^2_S<3\chi(S)$, then the Albanese map is a pencil of hyperelliptic curves of genus $\le 3$.
\end{enumerate}
\end{prop}
\begin{proof}
By the Severi inequality, the Albanese map of $S$  is a pencil $a\colon S\to B$, where $B$  has genus $b$. By the slope inequality we have:
\begin{equation}\label{eq:slope}
\frac{K^2_S}{\chi(S)}\ge \lambda(a)\ge \frac{4(g-1)}{g}\ge 2,
\end{equation}
where, as usual, $g$ denotes the genus of a general fiber of $a$.
If $K^2_S=2\chi(S)$, then all the  inequalities in \eqref{eq:slope} are equalities, hence $g=2$ and $b=1$. 

If $K^2_S<\frac 8 3 \chi(S)$, then \eqref{eq:slope} gives  $g=2$.

If $K^2_S<3\chi(S)$, then \eqref{eq:slope} gives  $g\le 3$. The general fiber of $a$ is hyperelliptic, since otherwise $\lambda(a)\ge 3$ by \cite[Lem. 3.1]{Konno_Trigonal}.
\end{proof}
Next we consider regular surfaces that have an irregular \'etale cover.

\begin{prop}\label{prop:K2<3}
Let $S$ be a minimal regular surface of general type. Then:
\begin{enumerate}
\item if $K^2_S<\frac 8 3\chi(S)$, then $S$ has no irregular \'etale cover;
\item if $K^2_S<3\chi(S)$, $S$ has an irregular \'etale cover iff it has a pencil of hyperelliptic curves of genus 3 with at least $4$ double fibers. 
\end{enumerate}
\end{prop}
\begin{proof}
Assume that $S'\to S$ is an irregular \'etale cover. Then by Lemma \ref{lem:consequence}, there is a pencil $f\colon S\to\pp^1$ such that the general fiber of $f$ has genus $\le 1+\frac{K^2_S}{4\chi(S)-K^2_S}$ and there are multiple fibers $m_1F_1,\dots m_kF_k$ such that $\sum_i\frac{m_i-1}{m_i}\ge 2$.
 Since by the adjunction formula a pencil of curves of genus 2 has no multiple fibers, it follows $g\ge 3$ and $K^2_S\ge \frac 8 3 \chi(S)$. This proves (i).

If $K^2_S<3\chi(S)$, then $g=3$ and the general fiber of $f$ is hyperelliptic (cf. proof of Proposition \ref{prop:2chi_char}). By the adjunction formula, the multiple fibers of $f$ are double fibers, hence   there are at least $4$ double fibers.

Conversely, assume that $f\colon S\to \pp^1$ is a pencil and that  $y_1, \dots y_4\in \pp^1$ are points such that the fiber of $f$ over $y_i$ is double. Let $B\to  \pp^1$ be the double cover branched on $y_1,\dots y_4$ and let $S'\to S$ be obtained from $B\to\pp^1$ by  taking base change with $f$ and normalizing.  The map $S'\to S$ is an \'etale double cover and by construction $S'$ maps onto the genus 1 curve $B$, hence $q(S')\ge 1$.
\end{proof}

\begin{remark} With some more work, it can be shown that Proposition \ref{prop:K2<3}, (ii) also holds for $K^2_S=3\chi(S)$ (\cite[Thm.1.1]{RitaMarg_JDG}). 

\end{remark}

\bigskip

\bigskip

\begin{minipage}{13cm}
\parbox[t]{6.5cm}{Margarida Mendes Lopes\\
Departamento de  Matem\'atica\\
Instituto Superior T\'ecnico\\
Universidade T{\'e}cnica de Lisboa\\
Av.~Rovisco Pais\\
1049-001 Lisboa, PORTUGAL\\
mmlopes@math.ist.utl.pt
 } \hfill
\parbox[t]{5.5cm}{Rita Pardini\\
Dipartimento di Matematica\\
Universit\`a di Pisa\\
Largo B. Pontecorvo, 5\\
56127 Pisa, Italy\\
pardini@dm.unipi.it}
\end{minipage}


\begin{thebibliography}{ABC00}
\bibitem[AK00]{Azumino} T.~Ashikaga, K.~Konno, {\em Global and local properties of pencils of algebraic curves}, in ``Algebraic Geometry 2000, Azumino", Adv. Stud. Pure Math. {\bf 36} (2000), 1--49 
\bibitem[Ba00]{Barja_threefolds} M.A.~Barja, {\em Lower bounds of the slope of fibred threefolds},
 Int. J. Math. {\bf 11}  4 (2000), 461--491.
\bibitem[Ba01]{Barja_bielliptic} M.A.~Barja, {\em On the slope of bielliptic fibrations},  Proc. Amer. Math. Soc. {\bf 129} (2001), 1899--1906. 
\bibitem[Be82]{Beauville_appendix} A.~Beauville, {\em L'in\'egalit\'e $p_g\ge 2q-4$ pour les surfaces de type g\'en\'eral}, appendix to \cite{Debarre_inequalities}.
\bibitem[Be96]{beauville_libro} A.~Beauville, {\em Complex Algebraic Surfaces}, second edition, London Mathematical Society student texts {\bf 34}, University Press, Cambridge (1996). 
\bibitem[BNP07]{bnp} M.A.~Barja, J.C.~Naranjo, G.P.~Pirola, {\em On the topological index of irregular surfaces},  J. Algebraic Geom.  {\bf 16}    no. 3  (2007), 435--458.

\bibitem[Bo73]{bombieri} E. Bombieri, {\em Canonical models of surfaces of general type}, Inst. Hautes \'Etudes Sci. Publ. Math. No. 42 (1973), 171--219. 
\bibitem[BPV84]{bpv} W.~Barth, C.~Peters, A.~Van de Ven, {\em Compact complex
surfaces}, Ergebnisse der Mathematik und ihrer Grenzgebiete, {\bf 3.}
Folge, Band {\bf 4}, Springer-Verlag, Berlin (1984).
\bibitem[BS08]{BarjaStoppino} M.A.~Barja, L.~Stoppino,  {\em Linear stability of projected canonical curves, with applications to the slope of fibred surfaces}, J.  Math. Soc. Japan {\bf 60} no. 1 (2008), 171--192.
\bibitem[BS09]{BarjaStoppino2} M.A.~Barja, L.~Stoppino,  {\em Slopes of trigonal fibred surfaces and of higher dimensional fibrations},   Ann. Sc. Norm. Super. Pisa Cl. Sci. (to appear), arXiv:0811.3305.
\bibitem[BZ01]{BarjaZucconi} M.A.~Barja, F.~Zucconi, {\em On the slope of fibred surfaces}.  Nagoya Math. J. {\bf 164} (2001), 103--131. 


\bibitem[Ca81]{Catanese_g2} F.~Catanese, {\em On a class of surfaces of general type}, in Algebraic Surfaces, CIME, Liguori (1981), 269--284.

\bibitem[Ca83]{Catanese_Ravello} F. Catanese, {\it  Moduli of surfaces of general type}, In: {\em Algebraic
geometry: open problems. Proc. Ravello 1982}, Springer--Verlag L.N.M. {\bf 997} (1983), 90--112.
\bibitem[Ca91]{catanese_irregular} F.~Catanese, {\em Moduli and classification of irregular Kaehler manifolds (and algebraic varieties) with Albanese general type fibrations}, Invent. Math. {\bf 104} no. 2 (1991), 263--289.
\bibitem[CC91]{CataneseCiliberto91} F.~Catanese,  C.~Ciliberto, {\em Surfaces with $p_g=q=1$}, Problems in the theory of surfaces and their classification (Cortona, 1988), Sympos. Math., XXXII, Academic Press, London (1991),  49--79.
\bibitem[CC93]{CataneseCiliberto93} F.~Catanese, C.~Ciliberto, {\em Symmetric products of elliptic curves and surfaces of general type with $p_g=q=1$}, J. Algebraic Geom. {\bf 2}  no.~3 (1993), 389--411.
\bibitem[CCM98]{CataneseCilibertoMLopes}  F.~Catanese, C.~Ciliberto,  M.~Mendes Lopes, {\em On the classification of irregular surfaces of general type with nonbirational bicanonical map}, Trans. Amer. Math. Soc. {\bf 350}  no.~1 (1998), 275--308.
\bibitem[CH06]{ChenHacon}
J.A.~Chen,C.D.~Hacon, {\em A surface of general type with $p_g=q=2$ and $K^2_X=5$},  Pac. J. Math. {\bf 223} no.~2  (2006), 219--228.
\bibitem[CM02]{CilibertoMLopes} C.~Ciliberto, M.~Mendes Lopes, {\em On surfaces with $p_g=q=2$ and non-birational bicanonical maps},  Adv. Geom. {\bf 2} no.~3 (2002), 281--300.
\bibitem[CMP07]{CiroRitaMarg} C.~Ciliberto, M.~Mendes Lopes, R.~Pardini, {\em Surfaces with $K^2<3\chi$ and finite fundamental group}, Math. Res. Lett. {\bf 14} no.~6 (2007), 1081--1098.  
\bibitem[CoH88]{CornalbaHarris} M. Cornalba, J. Harris, {\em Divisor classes associated to families of stable varieties, with applications to the moduli space of curves}, Ann. Sci. \'Ecole Norm. Sup., (4) {\bf 21}  no.~3 (1988), 455--475.
\bibitem[CoS08]{CornalbaStoppino} M.~Cornalba, L.~Stoppino, {\em A sharp bound for the slope of double cover fibrations}
 Michigan Math. J. {\bf 56} 3 (2008), 551--561. 

\bibitem[CPi06]{PirolaCausin} A.~Causin, G.P.~Pirola, {\em Hermitian matrices and cohomology of K\"ahler varieties}, Manuscripta Math. {\bf 121}  no.~2 (2006), 157--168.

\bibitem[CPo09]{CarnovalePolizzi} G.~Carnovale, F.~Polizzi, {\em The classification of surfaces with $p_g=q=1$ isogenous to a product of curves}, Adv. Geom. {\bf 9} no.~2 (2009), 233--256.

 
\bibitem[De82]{Debarre_inequalities} O. Debarre, {\em In\'egalit\'es num\'eriques pour les surfaces de type g\'en\'eral}, with an appendix by A. Beauville, Bull. Soc. Math. France {\bf 110} 3 (1982),  319--346.
\bibitem[Fu78]{Fujita_kahler} T.~Fujita, {\em On K\"ahler fibre spaces over curves}, J. Math. Soc. Japan {\bf 30} (1978), 779Ð-794.
\bibitem[Gi77]{gieseker} D.~Gieseker, {\em Global moduli for surfaces of general type},  Invent. Math. {\bf 43} (1977),  233--282.
\bibitem[GL87]{GreenLazarsfeld1}
M. Green, R. Lazarsfeld,
{\em Deformation  theory, generic vanishing theorems,
and some conjectures of Enriques, Catanese and Beauville},
 Invent. Math.
 {\bf 90} (1987),  389--407.

\bibitem[GL91]{GreenLazarsfeld2}
M. Green, R. Lazarsfeld,
{\em Higher obstruction to deforming cohomology groups of line bundles},
 Jour. Amer. Math. Soc.
  {\bf 4} (1991),  87--103.



\bibitem[Hi83]{Hirzebruch_lines} F.~Hirzebruch, {\em Arrangements of lines and algebraic surfaces}, In: Arithmetic and Geometry, vol. II, Prog. Math. {\bf 36}  Birkh\"auser (1983), 113--140.

\bibitem[Ho77]{Horikawa_genus2} E.~Horikawa, {\em On algebraic surfaces with pencils of curves of genus 2}, 
Complex Anal. algebr. Geom., Collect. Pap. dedic. K. Kodaira (1977), 79-90.

\bibitem[Ho81]{HorikawaV} E. Horikawa, {\em Algebraic surfaces of general type with small $c_1^2$, V}, J. Fac. Sci. Univ. Tokyo Sect. IA Math. {\bf 28} No. 3 (1981), 745--755.

\bibitem[HP02]{HaconPardini} C.D.~Hacon, R.~Pardini, {\em Surfaces with $p_g=q=3$}, Trans. Amer. Math. Soc. {\bf 354} no.~7 (2002), 2631-2638. 


\bibitem[HR98]{Ramanan}  A.Hirschowitz, S.~Ramanan,  {\em New evidence for Green's conjecture on syzygies of canonical curves},  Ann. Sci. \'Ecole Norm. Sup.,  (4) {\bf 31}  no.~2 (1998), 144--152.
 
\bibitem[Is83]{ishida} M.~Ishida, {\em The irregularity of Hirzebruch's examples of surfaces of general type with $c_1^2=3c_2$}, Math. Ann. {\bf 262} (1983), 407--420.



\bibitem[Ko93]{Konno_SNS} K.~Konno, {\em Non-hyperelliptic fibrations of small genus and certain irregular canonical surfaces}, 
Ann. Sc. Norm. Super. Pisa Cl. Sci. (4)  {\bf 10} (1993), 575--595.
\bibitem[Ko96a]{Konno_Trigonal} K.~Konno, {\em A lower bound of the slope of trigonal fibrations}, Internat. J. Math., 7 no.~1 (1996), 19--27.
\bibitem[Ko96b]{Konno_Severi} K. Konno, {\em Even canonical surfaces with small $K^2$,  III},  Nagoya Math. J.  {\bf 143} (1996), 1--11.

\bibitem[Ko99]{Konno_Clifford} K.~Konno, {\em Clifford index and the slope of fibered surfaces}, J. Algebraic Geom. {\bf 8} (1999), 207--220.
\bibitem[Ma03]{Manetti_Severi} M. Manetti, {\em Surfaces of Albanese general type and the Severi conjecture}, Math. Nach. {\bf 261--262} (2003), 105-122.

\bibitem[Mi73]{miyaoka}Y.~Miyaoka, {\em The maximal number of quotient singularities on surfaces with given numerical invariants}, Math. Ann., {\bf268} (1973), 159 --171.
\bibitem[MiP08]{MistrettaPolizzi} E.~Mistretta, F.~Polizzi, {\em Standard isotrivial fibrations with $p_g=q=1$}, J.  of Pure and Applied Algebra (to appear).
\bibitem[Mo96]{Moriwaki} A.~Moriwaki, {\em A sharp slope inequality for general stable fibrations of curves}, J. reine angew. Math. {\bf 480} (1996), 177--195. 


\bibitem[MP07]{RitaMarg_JDG} M.~Mendes Lopes, R.~Pardini, {\em On the algebraic fundamental group of surfaces with $K^2\le 3\chi$}, J. Differential Geom. {\bf 12}  no. 2 (2007), 177--188.
\bibitem[MP08a]{RitaMarg_JEMS} M.~Mendes Lopes, R.~Pardini, {\em Numerical Campedelli surfaces with fundamental group of order 9}, J.E.M.S. {\bf 10} (2) (2008), 457--476.
\bibitem[MP08b]{RitaMarg_advances} M.~Mendes Lopes, R.~Pardini, {\em On surfaces with $p_g = 2q-3$}, Adv.  Geom.  (to appear), arXiv:0811.0390.
\bibitem[MP09]{RitaMarg_Severi} M.~Mendes Lopes, R.~Pardini, {\em Severi type inequalities for irregular surfaces with ample canonical class}, Comm. Math. Helv. (to appear), arXiv:0904.1004.


\bibitem[MPR09]{MilesRitaMarg} M.~Mendes Lopes, R.~Pardini, M.~Reid, {\em  Campedelli surfaces with fundamental group of order 8}, Geom. Ded. {\bf 139} (1) (2009), 49--55. 

\bibitem[Oh92]{Ohno} K.~Ohno, {\em Some inequalities for minimal fibrations of surfaces of general type over curves}, 
 J. Math. Soc. Japan {\bf 44} 4 (1992), 643--666. 
\bibitem[Pa91]{ritaabel}    R. Pardini,
            \emph{Abelian covers of algebraic varieties},
                J. reine angew. Math. {\bf 417} (1991), 191--213.
\bibitem[Pa05]{Rita_Severi} R. Pardini, {\em The Severi inequality $K^2\ge 4\chi$ for surfaces of maximal Albanese dimension},  Invent. Math. {\bf 159}  3 (2005), 669--672. 
\bibitem[Pe81]{Persson} U.~Persson,
{\em Chern invariants of surfaces of general type},  Compos. Math. {\bf 43} no. 1 (1981), 3--58.
\bibitem[Pig08]{Pignatelli} R.~Pignatelli, {\em Some (big) irreducible components  of the moduli space of minimal surfaces of general type with $p_g=q=1$}, Atti Accad. Naz. Lincei Cl. Sci. Fis. Mat. Natur. Rend. Lincei (9) Mat. Appl. (to appear) .
\bibitem[Pir02]{Pirola} G.P.~Pirola, {\it Algebraic surfaces with $p_g=q=3$ and no irrational pencils}, Manuscripta Math. {\bf 108} no.~2 (2002), 163--170.
\bibitem[Pir09]{Pirola_q5}  G.P.~Pirola, personal communication.
\bibitem[Po08]{Polizzi1} F.~Polizzi, {\em On surfaces of general type with $p_g=q=1$ isogenous to a product of curves},  Communications in Algebra {\bf 36} no. 6 (2008),  2023--2053.
\bibitem[Po09]{Polizzi2} F.~Polizzi, {\em Standard isotrivial fibrations with $p_g=q=1$},  J. Algebra {\bf 321} (2009), 1600--1631.
\bibitem[PP08]{PareschiPopa} G.~Pareschi, M.~ Popa, {\em Strong generic vanishing and a higher dimensional Castelnuovo-de Franchis inequality}, arXiv:0808.2444.


\bibitem[Re78]{Miles_pi1} M. Reid, {\em $\pi_1$ for surfaces with small $c_1^2$}, Springer--Verlag L.N.M.
{\bf 732}  (1978), 534--544.

\bibitem[Ri07]{Rito1} C.~Rito, {\em On surfaces with $p_g=q=1$ and non-ruled bicanonical involution}, Ann. Scuola Norm. Sup. Pisa Cl. Sci.(5)  {\bf 6} no. 1 (2007), 81-102.
\bibitem[Ri08a]{Rito2} C.~Rito, {\em On equations of double planes with $p_g=q=1$},  Mathematics of Computation (to appear), arXiv:0804.2222.
\bibitem[Ri08b]{Rito3} C.~Rito, {\em Involutions on surfaces with $p_g=q=1$},  Collectanea Mathematica (to appear), arXiv:0805.4513.
\bibitem[Ro07]{Roulleau} X. Roulleau, {\em L'application cotangente des surfaces de type g\'en\'eral}, Th\`ese de doctorat, Universit\'e d'Angers (2007).
\bibitem[Sa66]{samuel} P.~Samuel, {\em Compl\'ements \`a un article de Hans Grauert sur la conjecture de Mordell},  Inst. Hautes \'Etudes Sci. Publ. Math. No. 29 (1966), 56--62.

\bibitem[Ser96] {Serrano} F. Serrano, {\it Isotrivial fibred surfaces}, Annali
di Matematica pura ed applicata (IV), {\bf CLXXI} (1996),  63--81.
\bibitem[Sev32]{Severi} F. Severi, {\em La serie canonica e la teoria delle serie principali di gruppi
di punti sopra una superficie algebrica}, Comm. Math. Helv. {\bf 4} (1932), 268--326.
\bibitem[SF00]{Stankova}  Z.E.~Stankova-Frenkel, {\em Moduli of trigonal curves}, J. Algebraic Geom. {\bf 9} (4) (2000), 607--662. 
\bibitem[So84]{sommese} A.J.~Sommese, {\em On the density of ratios of Chern numbers of algebraic surfaces}, Math. Ann. {\bf 268} (1984), 207-221.

\bibitem[St08]{Stoppino} L. Stoppino, {\em Slope inequalities for fibred surfaces via GIT}, Osaka Journal of Mathematics, Vol. 45, No. 4 (2008).

\bibitem[V05]{Voisin} C.~Voisin,  {\em Green's canonical syzygy conjecture for generic curves of odd genus}, Compos. Math.  {\bf 141}    no.~5 (2005), 1163--1190.
\bibitem[Xi87a]{Xiao_slope} G.~Xiao, {\it Fibered algebraic surfaces with
low slope}, Math. Ann. {\bf 276}
(1987), no. 3, 449--466. 
\bibitem[Xi87b]{Xiao_hyperell} G. Xiao, {\it Hyperelliptic surfaces of general type with $K^2<4\chi$}, Manuscripta  Math.  {\bf 57}
(1987), 125--148.
\bibitem[Ya77]{yau}
S.-T.~Yau, {\em Calabi's conjecture and some new results in algebraic geometry},  Proc. Nat. Acad. Sci. U.S.A.  {\bf 74}  no. 5  (1977), 1798--1799.
\bibitem[Zu03]{Zucconi} F.~Zucconi, {\em Surfaces with $p_g=q=2$ and an irrational pencil},
Canad. J. Math. {\bf 55} no. 3 (2003), 649-672.
\end{thebibliography}
\end{document}